\documentclass[12pt]{amsart}
\usepackage{amssymb}
\usepackage{mathtools}
\usepackage[english, french]{babel}
\usepackage{epsfig}
\usepackage{mathptmx}
\usepackage{amsmath,amsfonts,amssymb}
\usepackage{mathtools}
\usepackage{mathrsfs}
\usepackage[all]{xy}
\usepackage{graphicx}
\usepackage{latexsym}
\usepackage{verbatim}

\numberwithin{equation}{section}

\def\Re{{\sf Re}\,}
\def\Im{{\sf Im}\,}

\newcommand{\D}{\mathbb D}

\newcommand{\de}{\partial}
\newcommand{\R}{\mathbb R}
\newcommand{\Ha}{\mathbb H}

\newcommand{\C}{\mathbb C}

\newcommand{\oD}{\overline{\mathbb D}}

\newcommand{\N}{\mathbb N}

\def\Re{{\sf Re}\,}
\def\Im{{\sf Im}\,}

\newcommand{\strip}{\mathbb{S}}

\newcommand{\UD}{\mathbb{D}}

\def\Re{{\sf Re}\,}
\def\Im{{\sf Im}\,}

\def\Re{{\sf Re}\,}
\def\Im{{\sf Im}\,}

\def\1#1{\overline{#1}}
\def\2#1{\widetilde{#1}}
\def\3#1{\widehat{#1}}
\def\4#1{\mathbb{#1}}
\def\5#1{\frak{#1}}
\def\6#1{{\mathcal{#1}}}

\def\Re{{\sf Re}\,}
\def\Im{{\sf Im}\,}

\newcommand{\mcite}[1]{\csname b@#1\endcsname}

\theoremstyle{theorem}

\def\Re{{\sf Re}\,}
\def\Im{{\sf Im}\,}

\def\Arg{{\rm Arg}}

\newtheorem{theorem}{Theorem}[section]
\newtheorem{lemma}[theorem]{Lemma}
\newtheorem{proposition}[theorem]{Proposition}
\newtheorem{corollary}[theorem]{Corollary}

\theoremstyle{definition}
\newtheorem{definition}[theorem]{Definition}
\newtheorem{example}[theorem]{Example}

\theoremstyle{remark}
\newtheorem{remark}[theorem]{Remark}

\numberwithin{equation}{section}

\title[Backward orbits and petals]{Backward orbits and petals of semigroups of holomorphic self-maps of the unit disc}

\author[F. Bracci]{Filippo Bracci$^\dag$}
\address{F. Bracci: Dipartimento di Matematica, Universit\`a di Roma ``Tor Vergata", Via della Ricerca
Scientifica 1, 00133, Roma, Italia.} \email{fbracci@mat.uniroma2.it}

\author[M. D. Contreras]{Manuel D. Contreras$^\ddag$}

\author[S. D\'{\i}az-Madrigal]{Santiago D\'{\i}az-Madrigal$^\ddag$}
\address{M. D. Contreras, S. D\'{\i}az-Madrigal: Camino de los Descubrimientos, s/n\\
Departamento de Matem\'{a}tica Aplicada~II and IMUS\\ Universidad de Sevilla\\ Sevilla,
41092\\ Spain.}\email{contreras@us.es} \email{madrigal@us.es}

\author[H. Gaussier]{Herv\'e Gaussier}
\address{H. Gaussier: Univ. Grenoble Alpes, CNRS, IF, F-38000 Grenoble, France}\email{herve.gaussier@univ-grenoble-alpes.fr}


\thanks{$^\dag\,$Partially supported by GNSAGA of INdAM}

\thanks{$^\ddag$ Partially supported by the \textit{Ministerio
de Econom\'{\i}a y Competitividad} and the European Union (FEDER) MTM2015-63699-P and  by \textit{La Consejer\'{\i}a de Econom\'{\i}a y Conocimiento de la Junta de Andaluc\'{\i}a}.}

\long\def\REM#1{\relax}

\begin{document}
\maketitle

\selectlanguage{english}
\begin{abstract}
We study the backward invariant set of one-parameter semigroups of holomorphic self-maps of the unit disc. Such a set is foliated in maximal invariant curves and its open connected components are petals, which are, in fact, images of Poggi-Corradini's type pre-models. Hyperbolic petals are in one-to-one correspondence with repelling fixed points, while only parabolic semigroups can have parabolic petals. Petals have locally connected boundaries and, except a very particular case, they are indeed Jordan domains. The boundary of  a petal contains the Denjoy-Wolff point and, except such a fixed point, the closure of a petal contains either no other boundary fixed point or a unique repelling fixed point.
We also describe petals in terms of geometric and analytic behavior of K\"onigs functions using divergence rate and universality of models. Moreover, we construct a semigroup having a repelling fixed point in such a way that the intertwining map of the pre-model is not regular.  
\end{abstract}

\tableofcontents

\section{Introduction}

One-parameter continuous semigroups of holomorphic self-maps  of $\D$---for short, holomorphic semigroups in $\D$---have been widely studied, see, {\sl e.g.}, \cite{Ababook89,BerPor78,Shobook01,Sis98}.
In this paper, we study the behavior of semigroups at the boundary from a dynamical point of view, with special attention to boundary regular (in particular repelling) fixed points, a subject that has been addressed in a number of recent papers \cite{ConDia05a,CoDiPo04,CoDiPo06,ElShZa08-1,ElShZa08,Gum14}. 

Given a semigroup $(\phi_t)$ in $\D$ and a point $z\in \D$, one can follow the ``backward'' trajectory up to a boundary point. The union of the backward trajectory and the forward trajectory of $z$ is a maximal invariant curve for the semigroup. In case the backward trajectory is defined for all negative times, it is called a {\sl backward orbit}.  

Backward orbits for discrete holomorphic iteration in the unit disc have been introduced in \cite{Pog00}, where Poggi-Corradini proved that for every repelling fixed point of a holomorphic self-map of $\D$ there exists a backward orbit with bounded hyperbolic step. Using such an orbit as basis for a suitable rescaling, Poggi-Corradini showed the existence of pre-models. Abstract backward orbits for discrete iteration have been studied by the first named author in \cite{BrTAMS}, with the aim of proving a conjecture of Cowen \cite{Cow84} about common boundary fixed points of commuting holomorphic maps. Later, Poggi-Corradini \cite{Pog03,Pog04} gave a systematic treatment of the subject, and, recently, Arosio \cite{Ar} focused on backward orbits and pre-models with a categorial point of view which holds also in higher dimension. The previous cited results can be clearly adapted to holomorphic semigroups in $\D$, and we collect them in Section \ref{Sec:regular-backward}. 

We point out that, in  \cite{Pog00}, the author proved that the intertwining map of a pre-model at a repelling fixed point is always quasi-conformal, and he gave a (rather complicated) example of a holomorphic self-map of $\D$ for which the intertwining map of the pre-model is not regular. In Section \ref{sec-nr}, using suitable localization of the hyperbolic distance, we construct a holomorphic {\sl semigroup} of $\D$ for which the intertwining map of the pre-model at a repelling fixed point is not regular. 

The core and the main novelties of the paper are contained in Section \ref{Sec:petals}. There we start considering the {\sl backward invariant} set $\mathcal W$ of a holomorphic semigroup of $\D$, namely, the union of those points for which the backward trajectories are defined for every negative times. This set is foliated in real analytic curves (which we call {\sl maximal invariant curves}) which are orbits along which the Cauchy problem for the infinitesimal generator of the semigroup has a maximal solution defined for every real time. The interior of $\mathcal W$, the boundary of $\mathcal W$ and its complement are all completely $(\phi_t)$-invariant. The connected components of the interior of $\mathcal W$ are called {\sl petals}. All petals contain the Denjoy-Wolff point of $(\phi_t)$ in their closure. We call a petal {\sl parabolic} if it contains only a fixed point of the semigroup---that is, the Denjoy-Wolff point---in its closure, {\sl hyperbolic} otherwise. The main new results in the paper can be summarized in the following theorem:

\begin{theorem}\label{main-intro}
Let $(\phi_t)$ be a holomorphic semigroup in $\D$, not an elliptic group, with Denjoy-Wolff point $\tau\in \overline{\D}$. Let $\Delta$ be a petal of $(\phi_t)$. Then
\begin{enumerate}
\item $\Delta$ is simply connected and $\tau\in\partial\Delta$,
\item $\phi_t(\Delta)= \Delta$ and $(\phi_t|_\Delta)$ is a continuous one-parameter group of $\Delta$.
\item $\partial \Delta$ is locally connected and, except a very particular case, $\Delta$ is a Jordan domain. There are, in fact, only 5 possible types of petals (see Proposition \ref{Prop:shape-of-petals} for a precise description).
\item $\overline{\Delta}\setminus \tau$ contains at most one boundary fixed point $\sigma\in \partial \D$ of $(\phi_t)$. If this is the case, that is, $\Delta$ is hyperbolic, then $\sigma$ is a repelling fixed point of $(\phi_t)$. 
\item If $\Delta$ is parabolic, then $(\phi_t)$ is a parabolic semigroup. 
\item There is a one-to-one correspondence between repelling fixed points of $(\phi_t)$ and hyperbolic petals. 
\item For every $z\in \Delta$, the curve $[0,+\infty)\mapsto \phi_t(z)$ is a backward orbit with bounded hyperbolic step. 
\item If $\Delta$ is hyperbolic and $\sigma\in \partial \Delta\setminus\{\tau\}$ is the (unique) repelling fixed point of $(\phi_t)$ contained in $\overline{\Delta}$, then, for every $z\in \Delta$ it holds $\lim_{t\to-\infty}\frac{1}{t}\log(1-\overline{\sigma}\phi_t(z))=-\lambda$, where $\lambda\in (-\infty, 0)$ is the repelling spectral value of $(\phi_t)$ at $\sigma$.
\end{enumerate}
\end{theorem}

As a corollary of the previous result, we prove that if $\sigma\in\partial \D$ is a boundary fixed point which is not regular (also called {\sl super-repelling} fixed point), then there exists at most one backward orbit of $(\phi_t)$ landing at $\sigma$ (see Proposition \ref{Prop:uno-super-orbit}). 

\smallskip

In case of an elliptic and starlike holomorphic self-map of $\D$, Poggi-Corradini  proved that every repelling fixed point corresponds to a sector in the image domain of the K\"onigs function with amplitude related to the boundary dilation coefficient at the fixed point. This result, and its converse, for every case of semigroups was proved with a direct, lengthy and rather complicated argument (which needs to consider  the type of semigroup case by case) by the second and third author in \cite{ConDia05a}.

The next aim of this paper is to give a simple comprehensive proof of those results (see  Theorem \ref{Thm:petals-koenigs} and Theorem \ref{Thm:koenigs-petals}). The idea is that, since on each hyperbolic petal the semigroup acts as a hyperbolic group,  the restriction of  the K\"onigs function on the petal is a {\sl holomorphic model} in the sense of \cite{AroBra16}. Then the {\sl rate of divergence}---which is a measure of how fast an orbit escapes to the Denjoy-Wolff point---of the semigroup on the petal has to be the same for the model via the K\"onigs function. The divergence rate ``upstairs'' is  essentially given by the repelling spectral value of the semigroup at the corresponding repelling fixed point,  and hence, ``downstairs'', this forces the shape of the image of the K\"onigs function (see Section \ref{Sec:petals-Koenigs} for details). 

In Section  \ref{Sec:Boundary Fixed Points and Koenigs Functions}, we use the previously proved results to characterize the analytic behavior of the K\"onigs function of the semigroup at boundary points.

Finally, in Section \ref{Sec:examples}, we provide several examples and in Section \ref{sec-nr} we construct an example of a holomorphic semigroup having a  pre-model at a repelling fixed point whose intertwining map is not regular at such a point.

\section{Preliminaries}

For all the statements without references, we refer the reader to, {\sl e.g.}, \cite{Ababook89}, \cite{CowMacbook95} or \cite{Shabook83}.

A continuous one-parameter semigroup $(\phi_t)$ of holomorphic self-maps of~$\UD$---a holomorphic semigroup of $\D$ for short---is a continuous homomorphism $t\mapsto \phi_t$ from the
additive semigroup $(\R_{\ge0}, +)$ of non-negative real numbers to the
semigroup $({\sf Hol}(\D,\D),\circ)$ of  holomorphic self-maps
of $\D$ with respect to composition, endowed with the
topology of uniform convergence on compacta.

If $\phi_{t_0}$ is an automorphism of $\D$ for some $t_0>0$, then $\phi_t$ is an automorphism of $\D$ for all $t\geq 0$ and the semigroup can be extended to a group.

If $(\phi_t)$ is not a group of hyperbolic rotations, namely, it does not contain elliptic automorphisms of $\D$, then there exists a unique point $\tau\in \oD$ such that $\phi_t$ converges uniformly on compacta to the constant map $z\mapsto \tau$. Such a point $\tau$ is called the {\sl Denjoy-Wolff point} of $(\phi_t)$. 

The semigroup is called {\sl elliptic} if $\tau\in  \D$. In such a case, $\phi_t'(\tau)=e^{-\lambda t}$ with $\lambda\in \C$, $\Re \lambda>0$. 

If $\tau\in\partial \D$, then the non-tangential limit $\angle\lim_{z\to \tau}\phi_t(z)=\tau$ for all $t\geq 0$ and $\angle\lim_{z\to \tau}\phi_t'(z)=e^{-\lambda t}$ for some $\lambda\geq 0$. In case $\lambda>0$ the semigroup is called {\sl hyperbolic}, while, if $\lambda=0$ the semigroup is {\sl parabolic}. The number $\lambda$ is called the {\sl dilation} (or the {\sl spectral value}) of $(\phi_t)$. 

Parabolic semigroups can be divided in two sub-types: a parabolic holomorphic semigroup in $\D$ is {\sl of positive hyperbolic step} if $\lim_{t\to+\infty}\omega(\phi_{t+1}(0), \phi_t(0))>0$ (here $\omega(z,w)$ is the hyperbolic distance in $\D$ between $z\in \D$ and $w\in \D$). Otherwise, it is called   {\sl of zero hyperbolic step}.

\begin{definition}
A point $\sigma\in \partial \D$ is a {\sl boundary fixed point} of $(\phi_t)$ if $\angle\lim_{z\to \sigma}\phi_t(z)=\sigma$ for all $t\geq 0$.

Moreover, a boundary fixed point $\sigma\in \partial \D$ is called a {\sl boundary regular fixed point}  if $\angle\lim_{z\to \sigma}\phi_t'(z)=e^{-\mu t}$ for some $\mu\in \R$ and for all $t\geq 0$. If $\mu<0$, the point $\sigma$ is a {\sl repelling fixed point} of $(\phi_t)$. In this case, number $\mu$ is called the {\sl repelling spectral value} of $(\phi_t)$ at $\sigma$.

A boundary fixed point which is not regular is called a {\sl super-repelling fixed point}.
\end{definition}

We point out for the reader convenience that the previous definition is not the standard definition of a boundary regular fixed point for a holomorphic self-map of $\D$, but it is equivalent to that (see \cite[Prop. 1.2.8]{Ababook89}), and it is enough for our aims. We also note that, apart from the trivial semigroup, the only boundary regular fixed points of a semigroup are the Denjoy-Wolff point (provided the semigroup is not elliptic) and repelling fixed points.

It is  known (see, \cite[Theorem 1]{CoDiPo04}, \cite[Theorem 2]{CoDiPo06}, \cite[pag. 255]{Sis85}, \cite{EliShobook10}) that  a point $\sigma\in \de \D$ is a boundary (regular) fixed point of $\phi_{t_0}$ for some $t_0>0$ if and only if it is a boundary (regular) fixed point of $\phi_t$ for all $t\geq 0$.

By Berkson-Porta's theorem \cite[Theorem~(1.1)]{BerPor78}, if $(\phi_t)$
is a holomorphic semigroup in $\D$, then $t\mapsto \phi_t(z)$
is real-analytic and there exists a unique holomorphic vector field
$G:\D\to \C$ such that $\frac{\de \phi_t(z)}{\de
t}=G(\phi_t(z))$ for all $z\in\UD$ and all~$t\ge0$. This vector field $G$---the  {\sl infinitesimal generator} of $(\phi_t)$---is {\sl semicomplete} in the sense that the associated Cauchy problem
\[
\begin{cases}
\frac{dx(t)}{dt}=G(x(t)),\\
x(0)=z,
\end{cases}
\]
has a  solution $x^z:[0,+\infty)\to \D$ for every $z\in \D$.
Conversely, any semicomplete holomorphic vector field in $\D$
generates a continuous one-parameter semigroup of holomorphic self-maps of $\D$.

Another key notion associated to  semigroups is that of holomorphic model.  
\begin{definition}
Let $(\phi_t)$ be a semigroup of holomorphic self-maps of $\D$. A {\sl (holomorphic) model} for $(\phi_t)$ is a triple $(\Omega, h, \Phi_t)$ such that $\Omega$ is an open subset of $\C$, $\Phi_t$ is a group of (holomorphic) automorphisms of $\Omega$, $h: \D \to h(\D)\subset \Omega$ is univalent on the image, $h\circ \phi_t= \Phi_t\circ h$ and 
\begin{equation}\label{absorbing}
\cup_{t\geq 0} \Phi_t^{-1}(h(\D))=\Omega.
\end{equation}
\end{definition}
The previous notion of holomorphic model was introduced in \cite{AroBra16}, where it was proved that every semigroup of holomorphic self-maps of any complex manifold admits a  holomorphic model, unique up to ``holomorphic equivalence''. Moreover, a model is ``universal'' in the sense that every other conjugation of the semigroup to a group of automorphisms factorizes through the model (see \cite[Section 6]{AroBra16} for more details).

Notice that given a model $(\Omega, h, \Phi_t)$ for a semigroup $(\phi_t)$ of holomorphic self-maps of $\D$, then $(\phi_t)$ is a group if and only if $h(\D)=\Omega$.

Holomorphic models always exist and are unique up to holomorphic equivalence of models. In what follows we denote by  $\Ha:=\{\zeta \in \C: \Re \zeta>0\}$, $\Ha^{-}:=\{\zeta\in \C: \Re \zeta<0\}$ and, given $\rho>0$, $\strip_\rho:=\{\zeta\in \C: 0<\Re \zeta<\rho\}$. We simply write $\strip:=\strip_1$. The following result sums up the results in \cite{AroBra16, Cow81}, see also \cite{Ababook89}.

\begin{theorem}\label{modelholo}
Let $(\phi_t)$ be a  semigroup in $\D$. Then
\begin{enumerate}
\item $(\phi_t)$ is the trivial semigroup if and only if $(\phi_t)$ has a holomorphic model $(\D, {\sf id}_\D, z\mapsto z)$.
\item $(\phi_t)$ is a {\sl group} of elliptic automorphisms with spectral value $i\theta$, for $\theta\in \R\setminus\{0\}$, if and only if $(\phi_t)$ has a holomorphic model $(\D, h, z\mapsto e^{-i\theta t}z)$.
\item  $(\phi_t)$ is elliptic, not a group, with spectral value $\lambda$, for $\lambda\in \C$ with $\Re \lambda>0$, if and only if $(\phi_t)$ has a holomorphic model  $(\C, h, z\mapsto e^{-\lambda t} z)$.
\item  $(\phi_t)$ is hyperbolic with spectral value $\lambda>0$ if and only if  it has a holomorphic model $(\strip_{\frac{\pi}{ \lambda}}, h, z\mapsto z+it)$.
\item  $(\phi_t)$ is parabolic of positive hyperbolic step if and only if it has a holomorphic model either of the form $(\Ha, h, z\mapsto z+it)$ or of the form $(\Ha^-, h, z\mapsto z+it)$.
\item $(\phi_t)$ is parabolic of zero hyperbolic step if and only if
it has a holomorphic model $(\C, h, z\mapsto z+it)$.
\end{enumerate}
\end{theorem}
The holomorphic models defined in the previous theorem are called {\sl canonical}. The function $h$ in the canonical model of $(\phi_t)$ is called the {\sl K\oe nigs function} of the semigroup. 

Finally, in this paper we will make use of Carath\'eodory's prime ends theory. We refer the reader to  Pommerenke's books \cite{Pombook75, Pombook92} and Collingwood and Lohwater's book \cite{ColLohbook66} for all non proven statements about it.

\section{Backward orbits and pre-models}\label{Sec:regular-backward}

For $z,w\in \D$, we let $\omega(z,w)$ be the hyperbolic distance in $\D$ of $z, w$.

\begin{definition}
Let $(\phi_t)$ be a semigroup in $\D$. A continuous curve $\gamma:[0,+\infty)\to \D$ is called a  {\sl backward orbit}\index{Backward orbit} if for every $t\in (0,+\infty)$ and for every $0\leq s\leq t$,
\[
\phi_s(\gamma(t))=\gamma(t-s).
\]
A backward orbit $\gamma$ is said to be a {\sl regular backward orbit}\index{Regular backward orbit} if
\[
V(\gamma):=\limsup_{t\to +\infty}\omega(\gamma(t), \gamma(t+1))<+\infty.
\]
We call $V(\gamma)$ the {\sl hyperbolic step}\index{Hyperbolic step of a regular backward orbit} of $\gamma$.
\end{definition}

\begin{remark}\label{Rem:back-model}
Let $(\phi_t)$ be a semigroup in $\D$ and let $\gamma:[0,+\infty)\to \D$ be a backward  orbit for $(\phi_t)$. Let $(\Omega, h, \psi_t)$ be the canonical model of $(\phi_t)$ given by Theorem \ref{modelholo}. For all $t\geq 0$, $h(\gamma(t))\in h(\D)$, hence, $\psi_t(h(\gamma(t)))=h(\phi_t(\gamma(t)))=h(\gamma(0))$ and 
\[
h(\gamma(t))=\psi_{-t}(h(\gamma(0))),\quad \hbox{for all $t\geq 0$}.
\]
In particular, if $(\phi_t)$ is elliptic, not a group, and $\lambda\in \C$, $\Re\lambda>0$ is its spectral value, then 
\[
h(\gamma([0,+\infty)))=\left\{ e^{-\lambda s}h(\gamma(0)):s\in \R
\right\}\cap\{w\in \C: |w|\geq |h(\gamma(0))|\}.
\]
 While, if $(\phi_t)$ is non-elliptic, 
\[
h(\gamma([0,+\infty)))=\{w\in \C:\, \Re w=\Re h(\gamma(0))\}\cap\{w\in \C: \Im w\leq \Im h(\gamma(0))\}.
\]
\end{remark}

The study of backward orbits of groups is particularly easy by direct computation and we leave the details for the reader:

\begin{proposition}\label{Prop:group-backward}
Let $(\phi_t)$ be a non-trivial group in $\D$ and let $\tau\in\overline{\D}$ be the Denjoy-Wolff point of $(\phi_t)$. Let $\gamma:[0,+\infty)\to \D$ be a backward orbit. Then $\gamma$ is regular. Moreover,
\begin{enumerate} 
\item if $(\phi_t)$ is elliptic, then either $\gamma(t)\equiv \tau$ or the image of $\gamma$ is the boundary of a hyperbolic disc centered at $\tau$.
\item If $(\phi_t)$ is hyperbolic and $\sigma\in \partial \D\setminus\{\tau\}$ is the other fixed point of $(\phi_t)$, then $\lim_{t\to +\infty}\gamma(t)=\sigma$ and there exists $\alpha\in (-\pi/2, \pi/2)$ such that  $\lim_{t\to +\infty}\Arg(1-\overline{\sigma}\gamma(t))=\alpha$. 
\item  If $(\phi_t)$ is parabolic, then $\lim_{t\to +\infty}\gamma(t)=\tau$ and  $|\tau-\gamma(t)|^2=R( 1-|\gamma(t)|^2)$ for some $R>0$ and for all $t\geq 0$. In particular, $\gamma(t)$ converges to $\tau$ tangentially.
\end{enumerate}
\end{proposition}

\begin{remark}\label{Rem:backward-hyper-direc}
It is easy to see that if $(\phi_t)$ is a hyperbolic group with Denjoy-Wolff point $\tau\in \partial \D$ and other fixed point $\sigma\in \partial \D\setminus\{\tau\}$, then for every $\alpha\in (-\pi/2,\pi/2)$ there exists $z_0\in \D$ such that $[0,+\infty)\ni t\mapsto \phi_{-t}(z_0)$ is a regular backward orbit which converges to $\sigma$ and such that $\lim_{t\to +\infty}\Arg(1-\overline{\sigma}\phi_{-t}(z_0))=\alpha$. 
\end{remark}

Now, we examine the case of semigroups which are not groups. We start with the following result:

\begin{lemma}\label{Lem:converge-no-back}
Let $(\phi_t)$ be a semigroup in $\D$ which is not a group. Let $\tau\in \overline{\D}$ be its  Denjoy-Wolff point.  Let $\gamma:[0,+\infty)\to \D$ be a backward orbit for $(\phi_t)$. Then, 
\begin{itemize}
\item either $\tau\in \D$ and $\gamma(t)\equiv \tau$ for all $t\geq 0$,
\item or there exists $\sigma\in \partial \D$ (possibly $\sigma=\tau$) such that $\sigma$ is a fixed point of $(\phi_t)$ and $\lim_{t\to +\infty}\gamma(t)=\sigma$.
\end{itemize}
\end{lemma}
\begin{proof}
Let $(\Omega, h, \psi_t)$ be the canonical model for $(\phi_t)$, where $\Omega$ is either $\C, \Ha, \Ha^-$ or the strip $\strip_\rho$ for some $\rho>0$ and either $\psi_t(w)=w+it$ or $\psi_t(w)=e^{-\lambda t}w$ for some $\lambda\in \C$ with $\Re \lambda>0$. Let $w_0:=h(\gamma(0))$. Let $t\geq 0$. By Remark \ref{Rem:back-model}, $h(\gamma(t))=\psi_{-t}(w_0)$ for all $t\geq 0$. In particular, by the form of $\psi_t$, it follows that either $\psi_{-t}(w_0)\equiv w_0=0$ (and hence $\gamma(t)\equiv \tau$) or $\lim_{t\to+\infty}\psi_{-t}(w_0)=\infty$ in $\C_\infty$. Therefore, by \cite[Lemma 2, p. 162]{Shabook83}, there exists $\sigma\in \partial \D$ such that 
\[
\lim_{t\to +\infty}\gamma(t)=\lim_{t\to +\infty}h^{-1}(h(\gamma(t))=\lim_{t\to +\infty}h^{-1}(\psi_{-t}(w_0))=\sigma.
\]
Now, let $s>0$. Then 
\[
\lim_{t\to +\infty}\phi_s(\gamma(t))=\lim_{t\to +\infty}\gamma(t-s)=\sigma.
\]
Therefore, by Lehto-Virtanen Theorem, $\angle\lim_{z\to \sigma}\phi_s(z)=\sigma$ for all $s\geq 0$, hence $\sigma$ is a boundary fixed point of $(\phi_t)$. 
\end{proof}

The proof of the next result follows quite directly from the corresponding results for the discrete case (see   \cite[Lemma 2.1]{Pog03},  \cite{BrTAMS} and \cite{ConDia05a}).

\begin{proposition}\label{Prop:converge-back}
Let $(\phi_t)$ be a semigroup in $\D$ which is not a group. Let $\tau\in \overline{\D}$ be its  Denjoy-Wolff point.  Let $\gamma:[0,+\infty)\to \D$ be a regular backward orbit for $(\phi_t)$. Then,
\begin{enumerate}
\item either $\tau\in \D$ and $\gamma(t)=\tau$ for all $t\in [0,+\infty)$,
\item or, there exists a unique point $\sigma\in \partial\D\setminus\{\tau\}$ such that $\gamma(t)$ converges to $\sigma$ non-tangentially. Moreover,  $\sigma$ is a repelling fixed point of $(\phi_t)$ with repelling spectral value $\lambda\geq -2V(\gamma)$,
\item or, $\tau\in \partial \D$ and $\lim_{t\to+\infty}\gamma(t)=\tau$. Moreover, in this case, the convergence of $\gamma$ to $\tau$ is tangential and $(\phi_t)$ is parabolic. 
\end{enumerate}
\end{proposition}

Therefore, every  regular non-constant backward orbit lands at a repelling fixed point or at the Denjoy-Wolff point, and the hyperbolic step of the orbit controls the repelling spectral value of the semigroup at the fixed point. Adapting an argument from \cite{Pog00}, one can easily prove the converse:

\begin{proposition}\label{Prop:Existence-backward}
Let $(\phi_t)$ be a semigroup, not a group, in $\D$. Let $\sigma\in \partial \D$. Assume  $\sigma$ is a repelling fixed point of $(\phi_t)$ with repelling spectral value $\lambda\in (-\infty, 0)$. Then  there exists a (non-constant) regular backward orbit $\gamma:[0,+\infty)\to \D$  for $(\phi_t)$ such that $\lim_{t\to +\infty}\gamma(t)=\sigma$ and with hyperbolic step  $V(\gamma)=-\frac{1}{2}\lambda$.
\end{proposition}

The previous results show that for every repelling fixed point of a semigroup, there exists a regular backward orbit converging to such a point. Moreover, every backward orbit lands at a boundary fixed point of the semigroup. To close the circle one might wonder whether every super-repelling fixed point is the landing point of a (non-regular) backward orbit. The answer is negative: as it is proved in \cite[Prop. 4.9]{BCD}, a super-repelling fixed point is the landing point of a backward orbit if and only if it is not the initial point of a maximal contact arc. In particular, by \cite[Prop. 4.2]{BrGu}, if $\sigma\in \partial \D$ is a super-repelling fixed point for the semigroup $(\phi_t)$ and $G$ is the associated infinitesimal generator, then $\angle\lim_{z\to \sigma}G(z)= 0$ implies that  there exist no  backward orbits for $(\phi_t)$ landing at $\sigma$.

Now we  introduce the notion of pre-model:

\begin{definition}\label{Def:premodelo}
Let $(\phi_t)$ be a semigroup in $\D$ which is not a group. Let $\sigma\in \partial \D$ be a repelling fixed point for $(\phi_t)$ with repelling spectral value $\lambda\in (-\infty,0)$. We say that the triple $(\D, g, \eta_t)$ is a {\sl pre-model} for $(\phi_t)$ at $\sigma$ if 
\begin{enumerate}
\item  $(\eta_t)$ is the unique hyperbolic group with Denjoy-Wolff point $-\sigma$, other fixed point $\sigma$ and spectral value $-\lambda$,
\item $g:\D \to \D$ is univalent, $\angle\lim_{z\to \sigma}g(z)=\sigma$, and $g$ is  semi-conformal at $\sigma$, {\sl i.e.}, 
\[
\angle\lim_{z\to \sigma}\Arg \frac{1-\overline{\sigma}g(z)}{1-\overline{\sigma}z}=0,
\]
\item $g\circ \eta_t = \phi_t \circ g$ for all $t \geq 0$.
\end{enumerate}
\end{definition}

\begin{remark}\label{Rem:spectro-hyper-pre}
If $(\D, g, \eta_t)$ is a pre-model for a semigroup $(\phi_t)$ at a repelling fixed point $\sigma\in \partial\D$ with repelling spectral value $\lambda\in (-\infty, 0)$, it follows that the repelling spectral value of $(\eta_t)$ at $\sigma$ is $\lambda$. Indeed, we have $\eta_t'(\sigma)=e^{-\lambda t}$ for all $t\geq 0$.
\end{remark}

The proof of the next theorem can be  adapted from the discrete case from \cite{Pog00}. We leave details to the reader.

\begin{theorem}\label{Thm:premodel}
Let $(\phi_t)$ be a semigroup in $\D$ which is not a group. Let $\sigma\in \partial \D$ be a repelling fixed point for $(\phi_t)$ with repelling spectral value $\lambda\in (-\infty,0)$. Then there exists a pre-model $(\D, g, \eta_t)$  for $(\phi_t)$ at $\sigma$. 
\end{theorem}

It can be proved directly arguing as in \cite{Pog00} that if $(\D, \tilde g, \eta_t)$ is another pre-model for $(\phi_t)$, then $\tilde g=g\circ T$, where $T$ is a hyperbolic automorphism of $\D$ fixing $\pm \sigma$ (and in particular $T\circ \eta_t=\eta_t\circ T$ for all $t\geq 0$). This also follows easily from our further construction, see Remark \ref{unico-modello}.

In general, given a pre-model $(\D, g, \eta_t)$ for $(\phi_t)$ at $\sigma$, the map $g$ is not regular at $\sigma$ (see Section~\ref{sec-nr}).

A  straightforward consequence of Theorem \ref{Thm:premodel} is the following:

\begin{corollary}\label{Cor:premodel-and-back}
Let $(\phi_t)$ be a semigroup, not a group, in $\D$. Let $\sigma\in \partial \D$ be a repelling fixed point of $(\phi_t)$ with repelling spectral value $\lambda\in (-\infty, 0)$. Let $(\D, g, \eta_t)$ be a pre-model for $(\phi_t)$ at $\sigma$.  Then, for every $\alpha\in (-\pi/2,\pi/2)$ there exists a (non-constant) backward orbit $\gamma:[0,+\infty)\to \D$ for $(\phi_t)$, with $\gamma([0,+\infty))\subset g(\D)$, converging to $\sigma$ such that 
\begin{equation}\label{Eq:pendiente-backward}
\lim_{t\to +\infty}\Arg (1-\overline{\sigma}\gamma(t))=\alpha.
\end{equation}
Conversely, if $\gamma:[0,+\infty)\to \D$ is a backward orbit converging  to $\sigma$, then $\gamma$ is regular if and only if it converges non-tangentially to $\sigma$. Moreover, in this case, $\gamma([0,+\infty))\subset g(\D)$ and there exists $\alpha\in (-\pi/2,\pi/2)$ such that \eqref{Eq:pendiente-backward} holds.
\end{corollary}

\section{Petals} \label{Sec:petals}

\begin{definition}
Let $(\phi_t)$ be a semigroup, not a group, in $\D$. Let 
\[
\mathcal W:=\cap_{t\geq 0}\phi_t(\D).
\]
We call $\mathcal W$ the {\sl backward invariant set}\index{Backward invariant set of a semigroup} of $(\phi_t)$.
We call each non-empty connected component of $\stackrel{\circ}{\mathcal W}$, the interior of $\mathcal W$, a {\sl petal} \index{Petal of a semigroup} for $(\phi_t)$. 
\end{definition}

Backward invariant sets are strictly related to invariant curves. We start with a definition,

\begin{definition}
Let $(\phi_t)$ be a semigroup, not a group, in $\D$, with Denjoy-Wolff point $\tau\in\overline{\D}$. Let $a\in [-\infty, 0)$. A continuous curve $\gamma:(a, +\infty)\to \D$ is called a {\sl maximal invariant curve}\index{Maximal invariant curve} for $(\phi_t)$ if $\phi_s(\gamma(t))=\gamma(t+s)$ for all $s\geq 0$ and $t\in (a,+\infty)$, $\lim_{t\to +\infty} \gamma(t)=\tau$ and there exists $p\in \partial \D$ such that $\lim_{t\to a^+}\gamma(t)=p$. We call $p$ the {\sl starting point}\index{Starting point of a maximal invariant curve} of $\gamma$.
\end{definition}

\begin{remark}\label{Rem:complete-inv-Jordan}
Let $a\in [-\infty, 0)$. Let $\gamma:(a, +\infty)\to \D$ be  a maximal invariant curve for a semigroup $(\phi_t)$ which is not a group and let $\lim_{t\to a^+}\gamma(t)=p\in \partial \D$. Then $\gamma$ is injective. Indeed, if not, say $\gamma(t)=\gamma(s)$ for some $a<s<t$. Then $\phi_{t-s}(\gamma(s))=\gamma(t)=\gamma(s)$, implying that $\gamma(s)=\tau\in \D$, the Denjoy-Wolff point of $(\phi_t)$. But then, for all $a<t_0<s$, $\phi_{s-t_0}(\gamma(t_0))=\gamma(s)=\tau$. Since $\phi_{s-t_0}$ is injective in $\D$, this forces $\gamma(t_0)=\tau$ for all $a<t_0<s$, a contradiction.

Let $h:(0,1)\to (a,+\infty)$ be any orientation preserving homeomorphism. Then, setting $\ell (0)=p$, $\ell(t):=\gamma(h(t))$ and $\ell(1)=\tau$, it follows that $\ell:[0,1]\to \overline{\D}$ is a Jordan arc (or a Jordan curve if $p=\tau$). 
\end{remark}

\begin{remark}\label{Rem:backward to complete invariant}
Clearly every non-constant backward orbit gives rise to a maximal invariant curve with starting point a fixed point by  flowing forward the backward orbit and reversing orientation. 
\end{remark}

\begin{proposition}\label{Prop:compl-inv-puntos}
Let $(\phi_t)$ be a semigroup, not a group, in $\D$ with Denjoy-Wolff point $\tau\in\overline{\D}$ and backward invariant set $\mathcal W$. Then for every $z_0\in \D\setminus\{\tau\}$ there exists a unique maximal invariant curve $\gamma:(a,+\infty)\to \D$ for $(\phi_t)$ such that $\gamma(0)=z_0$. Moreover,  $z_0\in \mathcal W$ if and only if $a=-\infty$. Also, denote by $p\in \partial \D$  the starting point of $\gamma$. Then
\begin{enumerate}
\item  if $z_0\not\in \mathcal W$ then $\gamma(t+a)=\phi_{t}(p)$  for all $t\in (0,+\infty)$, in particular $p$ is not a fixed point for $(\phi_t)$, and $\gamma(t)\not\in \mathcal W$ for all $t> a$.
\item if $z_0\in \mathcal W$ then $[0,+\infty)\ni t\mapsto \gamma(-t)$ is a backward orbit and $p$ is a boundary fixed point of $(\phi_t)$. Also,  $z_0\in \stackrel{\circ}{\mathcal W}$ if and only if $\gamma(t)\in \stackrel{\circ}{\mathcal W}$ for all $t\in \R$.
\end{enumerate}
In particular, $\phi_t(\stackrel{\circ}{\mathcal W})=\stackrel{\circ}{\mathcal W}$, $\phi_t(\mathcal W\setminus \stackrel{\circ}{\mathcal W})=\mathcal W\setminus \stackrel{\circ}{\mathcal W}$ and $\phi_t(\D\setminus \mathcal W)\subset \D\setminus \mathcal W$ for all $t\geq 0$.
\end{proposition}

\begin{proof}
First we show that if there exists a maximal invariant curve $\gamma$ for $(\phi_t)$ such that $\gamma(0)=z_0$, then it is unique. Assume $\tilde{\gamma}:(\tilde a,+\infty)\to \D$ is another maximal invariant curve such that $\tilde{\gamma}(0)=z_0$, for some $\tilde a<0$. Then for $t\geq 0$, 
\[
\gamma(t)=\phi_t(\gamma(0))=\phi_t(z_0)=\phi_t(\tilde\gamma(0))=\tilde\gamma(t). 	
\]
For $\max\{a, \tilde a\}<t<0$, 
\[
\phi_{-t}(\gamma(t))=\gamma(0)=z_0=\tilde\gamma(0)=\phi_{-t}(\tilde\gamma(t)).
\]
Since $\phi_{-t}$ is injective, we have $\gamma(t)=\tilde\gamma(t)$. If $\tilde a>a$, then $\lim_{t\to \tilde a^+}\tilde\gamma(z)=\lim_{t\to \tilde a^+}\gamma(t)\in \D$, a contradiction. Similarly, if $a<\tilde a$, we get a contradiction. Then, if exists, such a maximal invariant curve is unique.

Now we construct a maximal invariant curve. Let $(\Omega, h, \psi_t)$ be the canonical model for $(\phi_t)$, where, $\Omega$ is either $\C, \Ha, \Ha^-$ or the strip $\strip_\rho$ for some $\rho>0$ and either $\psi_t(w)=w+it$ or $\psi_t(w)=e^{-\lambda t}w$ for some $\lambda\in \C$ with $\Re \lambda>0$. Let  $w_0:=h(z_0)$ and
\[
a:=\inf\{t<0: \psi_{-t}(w_0)\in h(\D)\}.
\]
By the geometry of $\Omega$ (spirallike or starlike at infinity) and the form of $\psi_t$, it follows that  $\psi_t(w_0)\in h(\D)$ for every $t>a$. Hence, we can well define
\begin{equation*}
\gamma(t):=h^{-1}(\psi_t(w_0)), \quad t\in (a,+\infty).
\end{equation*}
Since $h^{-1}\circ \psi_t=\phi_t\circ h^{-1}$ on $h(\D)$, it is easy to see that $\phi_s(\gamma(t))=\gamma(s+t)$ for all $s\geq 0$ and $t\in (a,+\infty)$. This implies at once that $\lim_{t\to +\infty}\gamma(t)=\tau$.

Assume that $a>-\infty$. Then $q:=\psi_{a}(w_0)\in \partial h(\D)\cap \C$ and $w_0=\psi_{-a}(q)$. By construction, the curve $(a,+\infty)\ni\mapsto \psi_{t-a}(q)$ is contained in $h(\D)$ and $\lim_{t\to a^+}\psi_{t-a}(q)=q$. By \cite[Lemma 2, p. 162]{Shabook83},  there exists $p\in \partial \D$ such that $\lim_{t\to a^+}h^{-1}(\psi_{t-a}(q))=p$. Hence,
\[
\lim_{t\to a^+}\gamma(t)=\lim_{t\to a^+}h^{-1}(\psi_t(w_0))=\lim_{t\to a^+}h^{-1}(\psi_t(\psi_{-a}(q)))=\lim_{t\to a^+}h^{-1}(\psi_{t-a}(q))=p.
\]
Moreover, for all $s\geq 0$ and $t> a$, 
\[
\phi_s(\gamma(t))=\phi_s(h^{-1}(\psi_t(w_0)))=\phi_s(h^{-1}(\psi_{t-a}(q)))=h^{-1}(\psi_{t+s-a}(q)).
\]
Hence, $\lim_{t\to a+}\phi_s(\gamma(t))=h^{-1}(\psi_{s}(q))\in \D$. Therefore, by Lehto-Virtanen Theorem,  for all $s>0$,
\[
\phi_s(p)= h^{-1}(\psi_{s}(q))=h^{-1}(\psi_{s+a}(w_0))=\gamma(s+a).
\]
It is clear that $\phi_s(p)\not\in\mathcal W$ for all $s>0$, hence, $\gamma(t)\not\in \mathcal W$ for all $t>a$. 

Assume now $a=-\infty$. Clearly, $[0,+\infty)\ni t\mapsto \gamma(-t)$ is a backward orbit for $(\phi_t)$, hence either it is constantly equal to $\tau\in \D$, which implies at once that $z_0=\tau$, against our hypothesis, or it converges to a point $p\in\partial \D$ which is fixed for $(\phi_t)$ by Proposition \ref{Prop:converge-back}. Now, let $t\in \R$ and $s\geq 0$. Then $\gamma(t)=\phi_s(\gamma(t-s))$, which implies that $\gamma(t)\in \mathcal W$ for all $t\in \R$, in particular, $z_0\in \mathcal W$. 

The previous argument shows that  for all $t\geq 0$
\begin{equation}\label{Eq:totally-invariant}
\phi_t(\mathcal W)=(\phi_t)^{-1}(\mathcal W)=\mathcal W. 
\end{equation}
Now, suppose $z_0\in \stackrel{\circ}{\mathcal W}$. Then there exists an open set $U\subset \mathcal W$ such that $z_0\in U$.  Let $z_1=\gamma(t_1)$ for some $t_1\in (a,+\infty)$. If $t_1>0$, then $\phi_{t_1}(U)\subset \mathcal W$ is an open neighborhood of $\phi_{t_1}(z_0)=\gamma(t_1)=z_1$, hence, $z_1\in \stackrel{\circ}{\mathcal W}$. If $t_1<0$, taking into account that $U\subset \phi_{t_1}(\D)$, it follows that $V:=(\phi_{t_1})^{-1}(U)$ is an open neighborhood of $z_1$. By \eqref{Eq:totally-invariant}, $V\subset \mathcal W$, hence $z_1\in \stackrel{\circ}{\mathcal W}$. 
\end{proof}

\begin{remark}
Let $(\phi_t)$ be a semigroup, not a group, in $\D$, with infinitesimal generator $G$. Let $\gamma:(a, +\infty)\to \D$ be a maximal invariant curve. Then $\gamma$ is the maximal solution to the Cauchy problem $\frac{d}{dt}x(t)=G(x(t)), x(0)=\gamma(0)$. It follows at once differentiating in $s$ the expression $\phi_s(\gamma(t))=\gamma(t+s)$. Conversely, by the uniqueness property of the solution to the Cauchy problem, given $z_0\in \D$ and $\gamma:(-\epsilon, +\infty)\to \D$, $\epsilon\in (0,+\infty]$ the maximal solution to the Cauchy problem $\frac{d}{dt}x(t)=G(x(t)), x(0)=z_0$, then $\gamma$ is a maximal invariant curve of $(\phi_t)$. One can use this point of view to prove in another way the previous proposition.
\end{remark}

\begin{remark}\label{Rem:trasla-curva-max}
Let $(\phi_t)$ be a semigroup, not a group, in $\D$. Let $z_0\in \D\setminus\{\tau\}$ and let   $\gamma:(a,+\infty)\to \D$ be the unique maximal invariant curve for $(\phi_t)$ such that $\gamma(0)=z_0$, $a<0$. Let $t_0\in (a,+\infty)$.  It is easy to check that $\tilde\gamma:(a+t_0,+\infty)\to \D$  (where, if $a=-\infty$, we set $a+t_0=-\infty$) is a maximal invariant curve for $(\phi_t)$ such that $\tilde\gamma(0)=\gamma(t_0)$. In other words,  for every $z\in \gamma((a,+\infty))$, the image of the maximal invariant curve for $(\phi_t)$ which values $z$ at time $0$ is $ \gamma((a,+\infty))$.
\end{remark}

\begin{remark}\label{Rem:unique-start-max-bd}
The uniqueness of a maximal invariant curve holds also at  the starting point  in case this is not a fixed point. More precisely, let $(\phi_t)$ be a semigroup, not a group. Let $\gamma:(a,+\infty)\to \D$ be a maximal invariant curve for $(\phi_t)$ with $a<0$ and starting point $\sigma\in\partial \D$. Assume $a>-\infty$. If $\tilde{\gamma}:(\tilde a,+\infty)\to \D$ is a maximal invariant curve for $(\phi_t)$, $\tilde a<0$, such that $\lim_{t\to \tilde a^+}\tilde\gamma(t)=\sigma$ then $\tilde a>-\infty$ and $\gamma((a,+\infty))=\tilde\gamma((\tilde a,+\infty))$.

Indeed, by Proposition \ref{Prop:compl-inv-puntos}, $a>-\infty$  implies $\sigma$ is not a fixed point of $(\phi_t)$ which in turn, by the same token, implies $\tilde a>-\infty$. Then, by Proposition \ref{Prop:compl-inv-puntos}(1), $\gamma(t+a)=\phi_t(\sigma)=\tilde{\gamma}(t+\tilde a)$ for all $t\geq 0$, from which the previous statement follows at once.
\end{remark}

Now, we turn our attention to petals:

\begin{proposition}\label{Prop:simple-prop-petals}
Let $(\phi_t)$ be a semigroup, not a group, in $\D$ with Denjoy-Wolff point $\tau\in \overline{\D}$. Let $\Delta$ be a petal for $(\phi_t)$. Then
\begin{enumerate}
\item $\Delta$ is simply connected,
\item $\phi_t(\Delta)=\Delta$ for all $t\geq 0$ and $(\phi_t|_\Delta)$ is a continuous group of automorphisms of $\Delta$,
\item  $\tau\in \partial \Delta$,
\item there exists $\sigma\in\partial \D\cap\partial \Delta$ (possibly $\sigma=\tau$) such for every $z\in \Delta$ the curve $[0,+\infty)\ni t\mapsto (\phi_t|_\Delta)^{-1}(z)$ is a regular backward orbit for $(\phi_t)$ which converges to $\sigma$. Moreover, $\sigma$ is a boundary regular fixed point and, if $\sigma=\tau$, then $(\phi_t)$ is parabolic.
\end{enumerate}
\end{proposition}

\begin{proof}
(1) Seeking for a contradiction, we assume $\Delta$ is not simply connected. Hence, there exists a Jordan curve $\Gamma\subset \Delta$ such that the bounded connected component of $\C\setminus \Gamma$  contains a point $z_0\not\in {\mathcal W}$, the  backward invariant set of $(\phi_t)$. By Proposition \ref{Prop:compl-inv-puntos}, there exists a maximal invariant curve $\gamma:(a,+\infty)\to \D$ such that $\lim_{z\to a^-}\gamma(t)=p\in \partial \D$ and $\gamma(0)=z_0$. Moreover, $\gamma(t)\not\in  {\mathcal W}$ for all $t>a$. But  $\gamma((a,+\infty))$ has to intersect $\Gamma$, a contradiction. Therefore, $\Delta$ is simply connected.

(2) By Proposition \ref{Prop:compl-inv-puntos}, $\phi_t(\Delta)\subset \stackrel{\circ}{\mathcal W}$ for all $t\geq 0$. Hence,  since $\Delta$ is  connected, if $z_0\in \Delta$ then  $[0,+\infty)\ni t\mapsto \phi_t(z_0)$ is contained in $\Delta$. That is, $\phi_t(\Delta)\subseteq \Delta$ for all $t\geq 0$. On the other hand, if $z_0\in \Delta\setminus\{\tau\}$, let $\gamma:(-\infty,+\infty)\to \Delta$ be the maximal invariant curve  such that $\gamma(0)=z_0$, given by Proposition \ref{Prop:compl-inv-puntos}. Hence, since $\Delta$ is connected and $\gamma((-\infty,+\infty))\subset \stackrel{\circ}{\mathcal W}$, it follows that $\gamma(t)\in \Delta$ for all $t\in\R$. Thus, $\phi_t(\gamma(t))=\gamma(0)=z_0$ for every $t\geq 0$. Namely,  
$\phi_{t}(\Delta)=\Delta$ for all $t\geq 0$. Therefore, $(\phi_t|_\Delta)$ is a continuous semigroup whose iterates are automorphisms of $\Delta$. Clearly, it extends to a continuous group $(\phi_t|_\Delta)$  of $\Delta$. 

(3) If $z\in \Delta$, then $\phi_t(z)\in \Delta$ for all $t\geq 0$ and $\phi_t(z)\to \tau$. Hence, $\tau\in \overline{\Delta}$. In particular, if $(\phi_t)$ is not elliptic, then $\tau\in \partial \Delta$. Assume $(\phi_t)$ is elliptic and $\tau\in \Delta$. Since $\Delta$ is simply connected, there exists a univalent map $f:\D\to \C$ such that $f(\D)=\Delta$ and $f(0)=\tau$. Hence, by (2), $(f^{-1}\circ \phi_t\circ f)$ is a group of $\D$ with Denjoy-Wolff point $0$ such that 
\[
1=|(f^{-1}\circ \phi_t\circ f)'(0)|=|\phi_t'(\tau)|.
\]
By the Schwarz Lemma, it follows that $(\phi_t)$ is a group, against our assumption. 

(4) Let $z\in \Delta$. Taking into account that $\phi_t|_\Delta$ is an automorphism of $\Delta$ for all $t\geq0$, and denoting by $k_\Delta$ the hyperbolic distance in $\Delta$, we have 
\begin{equation*}
\begin{split}
\omega((\phi_t|_\Delta)^{-1}(z), (\phi_{t+1}|_\Delta)^{-1}(z))&\leq k_\Delta((\phi_t|_\Delta)^{-1}(z), (\phi_{t+1}|_\Delta)^{-1}(z))\\&= k_\Delta(z, (\phi_{1}|_\Delta)^{-1}(z))<+\infty.
\end{split}
\end{equation*}
Moreover, $\phi_s((\phi_t|_\Delta)^{-1}(z))=(\phi_{t-s}|_\Delta)^{-1}(z)$ for all $0\leq s\leq t$. Therefore, $[0,+\infty)\ni t\mapsto (\phi_t|_\Delta)^{-1}(z)$ is a regular backward orbit for $(\phi_t)$. Hence, by Proposition \ref{Prop:converge-back}, there exists $\sigma\in \partial \D\cap \partial \Delta$ such that $\lim_{t\to+\infty}(\phi_t|_\Delta)^{-1}(z)=\sigma$---and, if $\sigma=\tau$ then $(\phi_t)$ is parabolic. 

Finally, if $w\in \Delta$, for all $t\geq 0$,
\[
\omega((\phi_t|_\Delta)^{-1}(z), (\phi_{t}|_\Delta)^{-1}(w))\leq k_\Delta((\phi_t|_\Delta)^{-1}(z), (\phi_{t}|_\Delta)^{-1}(w))= k_\Delta(z, w)<+\infty,
\]
hence  $\lim_{t\to+\infty}(\phi_t|_\Delta)^{-1}(w)=\sigma$.
\end{proof}

The previous result allows to give the following definition:

\begin{definition}
Let $(\phi_t)$ be a semigroup, not a group, in $\D$ with Denjoy-Wolff point $\tau\in \overline{\D}$. Let $\Delta$ be a petal for $(\phi_t)$. We say that $\Delta$ is a {\sl hyperbolic petal}\index{Hyperbolic petal of a semigroup} if $\partial \Delta$ contains a repelling fixed point of $(\phi_t)$. Otherwise, we call  $\Delta$  a {\sl parabolic petal}.\index{Parabolic petal of a semigroup} 
\end{definition}

\begin{remark}\label{Rem:petalo-parab-parab}
By Proposition \ref{Prop:simple-prop-petals}, a petal  contains no (inner) fixed points of the semigroup. Moreover, only parabolic semigroup can have parabolic petals.
\end{remark}

Our aim now is to describe the boundary of petals. We start with the following result which, taking into account Remark \ref{Rem:backward to complete invariant}, shows that backward orbits which land at the Denjoy-Wolff point have to be contained in the closure of a parabolic petal:

\begin{proposition}\label{Prop:back-to-para-pet}
Let $(\phi_t)$ be a non-elliptic semigroup in $\D$ with Denjoy-Wolff point $\tau\in\partial \D$. Let $\gamma:(-\infty,+\infty)\to \D$ be a maximal invariant curve for $(\phi_t)$ such that $\lim_{t\to -\infty}\gamma(t)=\tau$. Let $J$ be the Jordan curve defined by $\gamma$, i.e., $J:=\overline{\gamma(\R)}$. Let $V$ be the bounded connected component of $\C\setminus J$. Then there exists a parabolic petal $\Delta$ of $(\phi_t)$ such that $V\subseteq \Delta$. In particular, $J$ is contained in the closure of a parabolic petal of $(\phi_t)$ and the semigroup is parabolic.
\end{proposition}
\begin{proof}
By Remark \ref{Rem:complete-inv-Jordan}, $J$ is a Jordan curve. Fix $\zeta_0\in V$.  Let $\eta:(a,+\infty)\to \D$ be the maximal invariant curve for $(\phi_t)$ such that $\eta(0)=\zeta_0$, $a<0$. Note that, since $J$ is the closure of the image of a maximal invariant curve which does not contain $\zeta_0$, we have $\eta((a,+\infty))\cap J=\emptyset$ by Remark \ref{Rem:trasla-curva-max}. Hence,  $\eta((a,+\infty))\subset V$. By Proposition \ref{Prop:compl-inv-puntos}, the initial point $w_0$ of $\eta$ belongs to $\partial \D$, hence, the only possibility is  $w_0=\tau$. In this case, again by Proposition \ref{Prop:compl-inv-puntos}, we have $\zeta_0\in \mathcal W$, the backward invariant set of $(\phi_t)$. By the arbitrariness of $\zeta_0$, this implies that $V\subset \stackrel{\circ}{\mathcal W}$. Therefore, there exists a petal $\Delta$ of $(\phi_t)$ such that $V\subseteq \Delta$. 

Finally, let $z\in V$. The curve $[0,+\infty)\ni t\mapsto (\phi_t|_\Delta)^{-1}(z)$  is a backward orbit of $(\phi_t)$ by Proposition \ref{Prop:simple-prop-petals}. By Remark \ref {Rem:backward to complete invariant}, it extends to a maximal invariant curve of $(\phi_t)$. Hence, for what we already proved, such a curve  converges to $\tau$.  It follows that $\Delta$ is parabolic, again by Proposition \ref{Prop:simple-prop-petals}(4).  
\end{proof}

Now,  we focus on the boundary of  petals. To this aim, and for the subsequent results, we need a lemma:

\begin{lemma}\label{Lem:boundary-inv-c}
Let $(\phi_t)$ be a  semigroup, not a group, in $\D$ with Denjoy-Wolff point $\tau\in\overline{\D}$. Suppose $\Delta\subset \D$ is a petal and $z_0\in \partial \Delta\cap (\D\setminus\{\tau\})$. Let $\gamma:(a,+\infty)\to \D$, $a<0$, be the maximal invariant curve of $(\phi_t)$ such that $\gamma(0)=z_0$. Then $\gamma((a,+\infty))\subset \partial \Delta$.
\end{lemma}
\begin{proof}
Let  ${\mathcal W}$ be the backward invariant set of $(\phi_t)$. Since $z_0\in \partial \Delta$, it follows that $z_0\not\in\stackrel{\circ}{\mathcal W}$. By Proposition \ref{Prop:compl-inv-puntos}, $\gamma(t)\not\in \stackrel{\circ}{\mathcal W}$ for all $t\in (a,+\infty)$. In particular, $\gamma(t)\not\in \Delta$ for all $t\in (a,+\infty)$.

Let $t\geq 0$ and fix $\delta>0$. Since $z_0\in\partial\Delta$,  there exists $w\in \Delta$ such that $\omega(z_0,w)<\delta$.  Then
\[
\omega(\phi_t(w), \gamma(t))=\omega(\phi_t(w), \phi_t(z_0))\leq \omega(w,z_0)<\delta.
\]
Since $\phi_t(w)\in \Delta$ for all $t\geq 0$ by Proposition \ref{Prop:simple-prop-petals}, by the arbitrariness of $\delta$, we have $\gamma(t)\in \partial \Delta$ for all $t\geq 0$. Now, let $t\in (a,0)$. Assume by contradiction that $\gamma(t)\not\in \partial \Delta$. Then there exists an open neighborhood $U$ of $\gamma(t)$ such that $U\cap \Delta=\emptyset$. Since $\phi_{-t}(\gamma(t))=\gamma(0)=z_0$, and $\phi_{-t}(U)$ is open, there exists $w\in \phi_{-t}(U)\cap \Delta$.  Since $\phi_{-t}|_\Delta$ is an automorphism of $\Delta$, it follows that $(\phi_{-t}|_\Delta)^{-1}(w)\in \Delta$. But $\phi_{-t}$ is injective, hence $(\phi_{-t}|_\Delta)^{-1}(w)\in U$, a contradiction. Hence, $\gamma(t)\in \partial \Delta$ for all $t\in (a,+\infty)$.  
\end{proof}

We first consider the case in which the boundary of a petal contains a maximal invariant curve starting at the Denjoy-Wolff point:

\begin{proposition}\label{Prop:para-pet-back}
Let $(\phi_t)$ be a parabolic semigroup in $\D$ with Denjoy-Wolff point $\tau\in\partial \D$. Let $\Delta\subset \D$ be a petal of $(\phi_t)$. Suppose there exists a maximal invariant curve $\gamma:(-\infty,+\infty)\to \D$ of $(\phi_t)$ such that $\lim_{t\to -\infty}\gamma(t)=\tau$. Let $J:=\overline{\gamma(\R)}$ (a Jordan curve) and let $V$ be the bounded connected component of $\C\setminus J$. Assume that  $J\cap \partial \Delta\cap \D\neq\emptyset$. Then $\Delta$ is parabolic, $\Delta=V$ and $J=\partial \Delta$. 

In particular, the boundary of a hyperbolic petal cannot contain maximal invariant curves with starting point $\tau$.
\end{proposition}
\begin{proof}
By Lemma \ref{Lem:boundary-inv-c}, $J\subset \partial \Delta$. Since $\Delta$ is connected and $J\cap \Delta=\emptyset$, either $\Delta\subseteq V$ or $\Delta\subseteq \D\setminus \overline{V}$. But, by Proposition \ref{Prop:back-to-para-pet}, $V$ is contained in a parabolic petal. Therefore, if $\Delta\subset V$, then, in fact, $V=\Delta$ and the statement holds.

Thus, assume by contradiction that $\Delta\subseteq U:=\D\setminus \overline{V}$.  Let $z_0\in J\cap \D$. We claim that there exists   an open neighborhood $A$ of $z_0$ such that $A\cap U\subset \Delta$.

Assume the claim is true. Let $\mathcal W$ be the backward invariant set of $(\phi_t)$. Note that $J\cap\D\subset \mathcal W\setminus \stackrel{\circ}{\mathcal W}$. But, we have: $A\cap U\subset \Delta\subset \mathcal W$ (by the claim), $A\cap V\subset \mathcal W$ (because $V$ is a petal) and  $J\cap \D\subset \mathcal W$  (by Proposition \ref{Prop:compl-inv-puntos}). Hence,  $A\subset \mathcal W$, which implies $z_0\in \stackrel{\circ}{\mathcal W}$, a contradiction.

We are left to prove the claim. 

If $\Delta=U$, then set $A=U$.  

Otherwise, there exists $\zeta_0\in \partial\Delta\cap U$. Let $\eta_0:(a,+\infty)\to \D$,  $a<0$, be the maximal invariant curve such that $\eta_0(0)=\zeta_0$. Let $C_0$ be its closure (it is a Jordan curve by Remark \ref{Rem:complete-inv-Jordan}). By Lemma \ref{Lem:boundary-inv-c}, $C_0\subset \partial \Delta$.  The Jordan curve $C_0$ divides $\D$ in two connected components, one of them, call it $B_0$, contains $\Delta$. Since $\zeta_0\not\in J$, we have $C_0\cap J\cap \D=\emptyset$ by Remark \ref{Rem:trasla-curva-max}, and hence, $\overline{V}\cap \D \subset B_0$. 

Let $p_0$ be the starting point of $\eta_0$. If $p_0=\tau$, then $B_0$ is the 
bounded connected component of $\D\setminus C_0$, hence, it is
contained in a parabolic petal by Proposition \ref{Prop:back-to-para-pet}. This implies $J\cap \D\subset \stackrel{\circ}{\mathcal W}$, a contradiction. Therefore, $p_0\neq \tau$. 

If $B_0=\Delta$, we are done setting $A=B_0$. Otherwise, there exists  $\zeta_1\in \partial\Delta\cap B_0$, and, arguing as before, we obtain another Jordan curve $C_1\subset \partial \Delta$, which does not intersect $C_0\cap \D$ and $J\cap \D$. Let $B_1$ be the connected component of $\D\setminus C_1$ which contains $\Delta$. By construction, such connected component contains $\overline{V}\cap \D$ and $C_0\cap \D$. 

Let $p_1$ be the starting point of the maximal invariant curve whose closure defines $C_1$. As before, $p_1\neq \tau$. By construction, $\Delta\subset G:=B_0\cap B_1\cap U$. If $G=\Delta$, we set $A=G$ and we are done. 

Otherwise, note that, by construction, $G$ is a simply connected domain whose  boundary  is given by $J\cup C_0\cup C_1\cup T$, where $T$ is a closed arc in $\partial \D$ with end points $p_0$ and $p_1$ (we set $T=\{p_0\}$ in case $p_0=p_1$). If there were a point $\zeta_2\in G\cap \partial \Delta$, the Jordan curve $C_2$, defined as before by $\zeta_2$, would divide $G$ into two connected components and $\Delta$ would belong to one of the two connected components. But this is impossible because $J, C_0, C_1\subset \partial \Delta$. The claim is proved. 
\end{proof}

Now we are in good shape to describe the boundary of petals:

\begin{proposition}\label{Prop:shape-of-petals}
Let $(\phi_t)$ be a semigroup, not a group, in $\D$ with Denjoy-Wolff point $\tau\in\overline{\D}$. Let $\Delta$ be a petal for $(\phi_t)$. Then one and only one of the following cases can happen:
\begin{enumerate}
\item $\tau \in \D$ and there exists a Jordan  arc $J$ with end points $\tau$ and a point  $p\in \partial \D$ such that $\partial \Delta= J\cup \partial \D$. In particular, in this case, $\Delta$ is the only petal of  $(\phi_t)$;
\item $\tau \in \partial \D$ and there exists a Jordan  arc $J$ with end points $\tau$ and a point  $p\in \partial \D\setminus\{\tau\}$ such that $\partial \Delta= J\cup A$ where, $A$ is an arc in $\partial \D$ with end points $\tau$ and $p$;
\item there exist two  Jordan arcs $J_1, J_2$ with end points $p\in \partial \D\setminus\{\tau\}$ and $\tau$ such that $J_1\cap J_2=\{\tau, p\}$ and $\partial \Delta=J_1\cup J_2$;
\item there exist two  Jordan arcs $J_1, J_2$ such that $J_j$ has end points $\tau$ and $p_j\in \partial \D\setminus\{\tau\}$, $j=1,2$, with $J_1\cap J_2=\{\tau\}$ and   $\partial \Delta=J_1\cup J_2\cup A$, where $A\subset \partial \D$ is a closed arc with end points $p_1$ and $p_2$;
\item $\tau\in \partial \D$ and there exists a Jordan curve $J$  such that  $J\cap\partial\D=\{\tau\}$, $\partial \Delta=J$ and $\Delta$ is the bounded connected component of $\C\setminus J$. In this case, $\Delta$ is a parabolic petal.
\end{enumerate}
In particular, $\partial \Delta$ is locally connected and, in cases (2), (3), (4) and (5), $\Delta$ is a Jordan domain.
\end{proposition}
\begin{proof}
Since $(\phi_t)$ is not a group by hypothesis, it follows that $\Delta\neq \D$  by Proposition \ref{Prop:simple-prop-petals}.  Hence, $\partial \Delta\cap \D\neq\emptyset$. Since $\Delta$ is simply connected, $\partial \Delta\cap \D$ cannot reduce to one point. Let $z_0\in  \partial \Delta\cap \D\setminus\{\tau\}$.

Let $\gamma:(a,+\infty)\to \D$ be the maximal invariant curve such that $\gamma(0)=z_0$, $a<0$, given by Proposition \ref{Prop:compl-inv-puntos}.
Let $J$ be the closure of $\gamma((a,+\infty))$. By Remark \ref{Rem:complete-inv-Jordan}, $J$ is a Jordan arc with end points $\tau$ and $p\in \partial \D$ (if $p=\tau$ then $J$ is a Jordan curve). By Lemma \ref{Lem:boundary-inv-c}, $J\subset \partial \Delta$. 

By Proposition \ref{Prop:para-pet-back}, if $p=\tau$ then we are in case (5). 

If $\partial \Delta\setminus J\subset \partial \D$ and $p\neq \tau$, 
then we are in case (1) or (2). In this case, if $\tau\in \D$ then $\partial \Delta$ is necessarily the union of $J$ with $\partial \D$ and, moreover, $\D\setminus\Delta=J$, which, since $J$ has no interior, implies that $\Delta$ is the only petal of $(\phi_t)$. On the other hand, if $\tau\in \partial \D$, and $p\neq \tau$, since $\Delta$ is connected, and $J$ disconnects $\D$ in two connected components by Jordan's Theorem, it follows that $\partial \Delta$ is  the union of $J$ with an arc in $\partial \D$ with end points $\tau$ and $p$.

Now, assume there exists $z_0'\in \partial \Delta\setminus J\cap \D$. We repeat the above argument, in order to obtain another Jordan arc (or Jordan curve) $J_2\subset\partial \Delta$ which contains $z_0'$ and $J_2$ is contained in $\D$ except, at most, the two ends points. By Remark  \ref{Rem:trasla-curva-max}, since the interior parts of both $J$ and $J_2$ are maximal invariant curves, and $z_0'\not\in J$, it follows that  $J\cap J_2\subset \partial \D\cup\{\tau\}$ (that is, $J$ and $J_2$ can have in common only the end points).
Let $p_2\in \partial\D$ be the initial point of $J_2$. By Proposition \ref{Prop:para-pet-back}, $p_2\neq \tau$.

We claim that $(\partial \Delta\setminus (J\cup J_2))\subset \partial \D$. Indeed, if this is not the case, one can find a point $w_0\in (\partial \Delta\setminus (J\cup J_2))\cap \D$. Repeating the above argument, we end up with another  Jordan curve $J_3$ whose interior does not intersect $J, J_2$. Recalling that $J, J_2, J_3$ have a common end point $\tau$ and another end point on $\partial \D\setminus\{\tau\}$, it is easy to see that we reach a contradiction. For instance, in case $\tau\in \D$, the Jordan curve $J\cup J_2$ divides $\D$ into two connected components, and $\Delta$ has to stay in one of them, call it $U$. Then $J_3$ is contained in $U$, and divides $U$ into two connected components, one whose boundary is $J, J_3$ and (possibly) an arc on $\partial \D$, and the other whose boundary is $J, J_2$ and (possibly) an arc on $\partial \D$. Since $\Delta$ is connected, it has to stay in one of the two components, say the first, but then $J_2$ can not be contained in $\partial \Delta$, a contradiction. The other cases are similar.

If $J_1$ and $J_2$ have the same end points, then we are in case (3). If $J_1$ has a different end point than $J$ (they both have $\tau$ as common end point), then we are in case (4). 
\end{proof}

Every case given by Proposition \ref{Prop:shape-of-petals} actually happens, as we will see in the last section.

\begin{proposition}\label{Prop:petal-no-other-fixed}
Let $(\phi_t)$ be a semigroup, not a group, in $\D$ with Denjoy-Wolff point $\tau\in\overline{\D}$. Let $\Delta$ be a petal for $(\phi_t)$.
\begin{enumerate}
\item If $\Delta$ is a hyperbolic petal, then there exists a repelling fixed point $\sigma\in \partial \Delta$ of $(\phi_t)$ such that $\partial \Delta\setminus \{\tau,\sigma\}$ does not contain any (repelling or super-repelling) boundary fixed point of $(\phi_t)$.
\item If $\Delta$ is a parabolic petal, then $\partial \Delta\setminus \{\tau\}$ does not contain any (repelling or super-repelling) boundary fixed point of $(\phi_t)$.
\end{enumerate}
\end{proposition}
\begin{proof}
Since $\Delta$ is simply connected by Proposition \ref{Prop:simple-prop-petals}, there exists a univalent function $g:\D\to \C$ such that $g(\D)=\Delta$.  Let $\psi_t:= g^{-1}\circ   \phi_t\circ g$, $t\geq 0$. Since $(\phi_t|_\Delta)$ is a continuous group of automorphisms of $\Delta$ by Proposition \ref{Prop:simple-prop-petals}, it follows that $(\psi_t)$ is a group in $\D$. Moreover, $\lim_{t\to +\infty}\phi_t(g(0))\to \tau\in \partial \Delta$, and hence $\psi_t(0)=g^{-1}(\phi_t(g(0)))$ can accumulate only on $\partial \D$. Therefore, $(\psi_t)$ is a non-elliptic group in $\D$.

By Proposition \ref{Prop:shape-of-petals}, $\partial \Delta$ is locally connected, hence, by Carath\'eodory Extension Theorem,  $g$ extends to a continuous  and surjective function, which we still denote by $g$, from $\overline{\D}$ to $\overline{\Delta}$. In particular, for all $p\in \partial \D$, 
\[
g(\psi_t(p))=\lim_{r\to 1^-}g(\psi_t((1-r)p)=\lim_{r\to 1^-}\phi_t(g((1-r)p)).
\]
By Lehto-Virtanen Theorem, it follows that the non-tangential limit of $\phi_t$ at $g(p)$ is $g(\psi_t(p))$, that is, 
\begin{equation}\label{Eq:model-ext-cont}
g(\psi_t(p))=\phi_t(g(p)) \quad \forall p\in \partial \D,\quad t\geq 0.
\end{equation}
Let $p\in \partial \D$ be such that $g(p)$ is a fixed point of $(\phi_t)$. Hence, by \eqref{Eq:model-ext-cont}, $g(\psi_t(p))=g(p)$ for all $t\geq 0$. We claim that this implies that $\psi_t(p)=p$ for all $t\geq 0$. Otherwise, the image $[0,+\infty)\ni t\mapsto \psi_t(p)$ would be an arc in $\partial \D$  where $g$ is constant. A contradiction.  

Since $(\psi_t)$ is a non-elliptic group in $\D$, it has at most two fixed points on $\partial \D$, hence, there exist at most two fixed points for $(\phi_t)$ on $\partial \Delta$. From this, (1) follows at once.

Now, assume that  $\Delta$ is parabolic. By Remark \ref{Rem:petalo-parab-parab}, $(\phi_t)$ is necessarily parabolic, and $\tau\in \partial \D$. Hence, by Proposition \ref{Prop:shape-of-petals}, $\Delta$ is a Jordan domain and  $g:\overline{\D}\to \overline{\Delta}$ is a homeomorphism. Now, by Proposition \ref{Prop:simple-prop-petals}(4), there exists a regular backward orbit $\gamma:[0,+\infty)\to \Delta$ such that $\lim_{t\to +\infty}\gamma(t)=\tau$. It is easy to see that $g^{-1}\circ \gamma$ is a backward orbit for $(\psi_t)$ which converges to $g^{-1}(\tau)$.  Since $\psi_t(0)=g^{-1}(\phi_t(g(0)))\to g^{-1}(\tau)$, it follows that $(\psi_t)$ has a backward orbit which converges to its Denjoy-Wolff point. Hence, $(\psi_t)$ is a parabolic group by Proposition \ref{Prop:group-backward}. Therefore, $(\psi_t)$ has only one fixed point on $\partial \D$, and so $(\phi_t)$ has a unique fixed point on $\partial \Delta$ which is $\tau$. 
\end{proof}

As a direct consequence of  Proposition \ref{Prop:petal-no-other-fixed} we have:

\begin{corollary}\label{Cor:frontera-interior-petalo}
Let $(\phi_t)$ be a semigroup, not a group, in $\D$ with Denjoy-Wolff point $\tau\in\overline{\D}$. Suppose $\Delta$ is a petal of $(\phi_t)$. Let $J\subset \partial \Delta\cap (\D\setminus\{\tau\})$ be a connected component. Then the closure of $J$ is a Jordan arc (or Jordan curve) with end points $\tau$ and $p_0\in\partial \D$. Moreover,
\begin{enumerate}
\item if $\Delta$ is hyperbolic and $\sigma\in\partial \Delta$ is the unique (repelling) fixed point of $(\phi_t)$ contained in $\overline{\Delta}$, then, either $p_0=\sigma$ or $p_0$ is not a fixed point of $(\phi_t)$.
\item if $\Delta$ is parabolic, then either $p_0=\tau$ or $p_0$ is not a fixed point of $(\phi_t)$.
\end{enumerate} 
\end{corollary}

The previous results show that the closure of every hyperbolic petal contains exactly one repelling fixed point, now we show the converse:

\begin{proposition}\label{Prop:petals-premodel}
Let $(\phi_t)$ be a semigroup, not a group, in $\D$ with Denjoy-Wolff point $\tau\in\overline{\D}$. Suppose $\sigma\in \partial \D$ is a repelling fixed point for $(\phi_t)$. Then there exists a unique hyperbolic petal $\Delta$ such that $\sigma\in\partial \Delta$. Moreover, for all $M>1$ there exists $\epsilon>0$ such that 
\[
S(\sigma, M)\cap \{\zeta\in \C:|\zeta-\sigma|<\epsilon\}\subset \Delta,
\]
where $S(\sigma,M):=\{z\in \D: |\sigma-z|<M(1-|z|)\}$ is a Stolz region in $\D$.

Also,  if $(\D, g, \eta_t)$ is a pre-model for $(\phi_t)$ at $\sigma$ then $g(\D)=\Delta$. 
\end{proposition}

\begin{proof} 
By Theorem \ref{Thm:premodel}, there exists a pre-model $(\D, g, \eta_t)$ for $(\phi_t)$ at $\sigma$. For all $t\geq 0$,
\begin{equation}\label{Eq:petal-grupo-premo}
\phi_t(g(\D))=g(\eta_t(\D))=g(\D).
\end{equation}
This implies that $g(\D)\subset \mathcal W$, the backward invariant set of $(\phi_t)$. Since $g(\D)$ is open and simply connected, there exists a petal $\Delta$ such that $g(\D)\subset \Delta$. Moreover, $\eta_{-t}(0)$ converges non-tangentially to $\sigma$ as $t\to +\infty$ by Proposition \ref{Prop:group-backward}, hence, since $\angle\lim_{z\to \sigma}g(z)=\sigma$, we have
\[
\lim_{t\to +\infty}g(\eta_{-t}(0))=\sigma,
\]
that is, $\sigma\in \overline{g(\D)}$. Since $\Delta$ contains no fixed points of $(\phi_t)$ by Remark \ref{Rem:petalo-parab-parab}, it follows that $\sigma\in \partial \Delta$, proving the first part of the statement.

Now we  show that $g(\D)=\Delta$. Let $z_0\in \Delta$.  By Proposition \ref{Prop:simple-prop-petals}, the curve $[0,+\infty)\ni t\mapsto (\phi_t|_\Delta)^{-1}(z_0)$ is a regular backward orbit for $(\phi_t)$ which converges to $\sigma$, and, by Proposition \ref{Prop:converge-back}, the convergence to $\sigma$ is non-tangential.
Therefore, by Corollary \ref{Cor:premodel-and-back}, $(\phi_t|_\Delta)^{-1}(z_0)\in g(\D)$ for all $t\geq 0$. In particular, $z_0\in g(\D)$ and hence $\Delta=g(\D)$ by the arbitrariness of $z_0$. Moreover, since $g$ is quasi-conformal at $\sigma$, for all $M>1$ there exists $\epsilon>0$ such that 
\[
S(\sigma, M)\cap \{\zeta\in \C:|\zeta-\sigma|<\epsilon\}\subset g(\D)=\Delta.
\]

Finally, we are left to show that $\Delta$ is the unique petal which contains $\sigma$ on its boundary. Assume by contradiction this is not the case and let $\Delta'$ be a petal different from $\Delta$ such that $\sigma\in \partial \Delta'$. Note that $\Delta\cap \Delta'=\emptyset$ (since they are different open connected components of the interior of the backward invariant set of $(\phi_t)$).  We claim that $\Delta'$ has to be a Jordan domain. Indeed, looking at Proposition \ref{Prop:shape-of-petals}, we see if $\Delta'$ is not a Jordan domain, then $\Delta'$ is the only petal of $(\phi_t)$, forcing $\Delta=\Delta'$.  Let $f:\D\to \C$ be univalent such that $f(\D)=\Delta'$. By Carath\'eodory's extension theorem, $f$ extends as a homeomorphism---which we still denote by $f$---from $\overline{\D}$ to $\overline{\Delta'}$. By Proposition \ref{Prop:simple-prop-petals}, $\tau\in \partial{\Delta'}$ and $(\phi_t|_{\Delta'})$ is a continuous  group of automorphisms of $\Delta'$. Hence,  arguing as in the proof of Proposition \ref{Prop:petal-no-other-fixed}, it is easy to see that $(f^{-1}\circ \phi_t\circ f)$ is a group in $\D$, with fixed points $f^{-1}(\tau)$ and $f^{-1}(\sigma)$. Therefore, $(f^{-1}\circ \phi_t\circ f)$ is a hyperbolic group and by Remark \ref{Rem:backward-hyper-direc}, it has a regular backward orbit $\gamma$ converging to $f^{-1}(\sigma)$. It is easy to see that $f\circ \gamma$ is a backward orbit for $(\phi_t)$ converging to $\sigma$ and it is regular because for all $t\geq 0$,
\[
\omega(f(\gamma(t)), f(\gamma(t+1)))\leq \omega(\gamma(t), \gamma(t+1)).
\]
By Proposition \ref{Prop:converge-back}, $f\circ \gamma$ converges to $\sigma$ non-tangentially. Corollary \ref{Cor:premodel-and-back} implies then that $f(\gamma([0,+\infty))\subset g(\D)=\Delta$. Hence, $\Delta\cap \Delta'\neq\emptyset$, a contradiction.
\end{proof}

\begin{remark}\label{unico-modello}
By the previous proposition, if $(\D, g, \eta_t)$ and $(\D,\tilde g, \eta_t)$ are pre-models for $(\phi_t)$ at $\sigma$, then $g^{-1}\circ \tilde g$ is an automorphism of $\D$ which commutes with $\eta_t$ for all $t\geq 0$. Therefore, by a direct computation (or see \cite{Hei}),  $g^{-1}\circ \tilde g$ is a hyperbolic automorphism of $\D$ fixing $\pm \sigma$.
\end{remark}

Given $z$ a point in a petal $\Delta$ of a semigroup $(\phi_t)$, it is well-defined $\phi_{t}|_{\Delta}(z)$ for all $t\in \R$. With a slight abuse of notation, we write $\phi_{t}(z)$ to denote  $\phi_{t}|_{\Delta}(z)$ for all $t\in \R$ when $z\in \Delta$.

\begin{proposition}\label{Prop:rate conv repelling}
Let $(\phi_t)$ be a semigroup in $\D$ with a repelling fixed point $\sigma\in \partial \D$, with repelling spectral value $\lambda \in (-\infty,0)$ and associated hyperbolic petal $\Delta$. Then, for all $z\in \Delta$,
\begin{equation}\label{Eq:rate conv repelling}
\lim_{t\to-\infty} \frac{1}{t}\log\left(1-\overline \sigma \phi_{t}(z)\right)=-\lambda.
\end{equation}
\end{proposition}
\begin{proof} Let $G$ be the infinitesimal generator of the semigroup. By \cite{CoDiPo06}, 
$$
-\lambda=\angle\lim_{z\to \sigma} \frac{G(z)}{z-\sigma}\in (0,\infty).
$$
Moreover,  the regular backward orbit $[0,+\infty)\ni t \mapsto \phi_{-t}(z)$ converges to $\sigma$ non-tangentially, so that
$$
-\lambda=\angle\lim_{t\to -\infty} \frac{G(\phi_{t}(z))}{\phi_{t}(z)-\sigma}, \quad z\in \Delta.
$$

Given $z\in \Delta$ and $t\in \R$, 
\begin{equation*}
\begin{split}
\int_{0}^{t}\frac{-\overline \sigma G(\phi_{s}(z))}{1-\overline{\sigma} \phi_{s}(z)}\, ds&=\left[\log(1-\overline{\sigma} \phi_{s}(z))\right]_{s=0}^{s=t}\\
&=\log(1-\overline{\sigma} \phi_{t}(z))-\log(1-\overline{\sigma} z).
\end{split}
\end{equation*}
 Then, using L'H\^opital's Rule and the non-tangential convergence, we obtain 
$$
\lim_{t\to -\infty}\frac{1}{t}\int_{0}^{t}\frac{-\overline \sigma G(\phi_{s}(z_{0}))}{1-\overline{\sigma} \phi_{s}(z_{0})}\, ds=\lim_{t\to +\infty}\frac{ G(\phi_{t}(z_{0}))}{ \phi_{t}(z_{0})-\sigma}=-\lambda.
$$ 
Hence, $\lim_{t\to-\infty} \frac{1}{t}\log\left(1-\overline \sigma \phi_{t}(z)\right)=-\lambda.$
\end{proof}

As a last result of this section we prove that a super-repelling fixed point can be the limit of at most one backward orbit:

\begin{proposition}\label{Prop:uno-super-orbit}
Let $(\phi_t)$ be a semigroup, not a group, in $\D$ with Denjoy-Wolff point $\tau\in\overline{\D}$. Suppose $\sigma\in \partial \D$ is a super-repelling fixed point of $(\phi_t)$. Assume $\gamma_j:[0,+\infty)\to \D$, $j=1,2$, are backward orbits of $(\phi_t)$ converging to $\sigma$. Then, either $\gamma_1([0,+\infty))\subseteq \gamma_2([0,+\infty))$ or $\gamma_2([0,+\infty))\subseteq \gamma_1([0,+\infty))$. 

In particular, up to re-parameterization, there is at most one maximal invariant curve for $(\phi_t)$ with starting point $\sigma$.
\end{proposition}
\begin{proof}
Suppose by contradiction that the statement is not true.  Let define $\eta_j(t)=\gamma_j(-t)$ for $t\geq 0$ and $\eta_j(t)=\phi_t(\gamma_j(0))$ for $t>0$, $j=1,2$. By Remark \ref{Rem:backward to complete invariant}, $\eta_1, \eta_2$ are maximal invariant curves for $(\phi_t)$. Hence, by Remark \ref{Rem:trasla-curva-max}, either $\eta_1((-\infty,+\infty))=\eta_1((-\infty,+\infty))$ or they are disjoint, and by our hypothesis, the latter case holds. Since $\lim_{t\to +\infty}\eta_j(t)=\tau$ and $\lim_{t\to -\infty}\eta_j(t)=\sigma$, $j=1,2$, by Remark \ref{Rem:complete-inv-Jordan} it follows that the closure of $\eta_1((-\infty,+\infty))\cup\eta_2((-\infty,+\infty))$, call it $J$, is a Jordan curve such that $J\cap \partial \D=\{\tau, \sigma\}$. Let $D$ be the bounded connected component of $\C\setminus J$. We claim that  $D\subset \mathcal W$, the backward invariant set of $(\phi_t)$.

Assuming the claim, it follows at once that $D$ is contained in a petal $\Delta$. But $\sigma\in \partial \Delta$, hence, by Proposition \ref{Prop:petal-no-other-fixed}, $\Delta$ is hyperbolic and $\sigma$ is repelling, contradiction. 

In order to prove the claim, let $z_0\in D$ and let $\eta:(a,+\infty)\to\D$ be the maximal invariant curve such that $\eta(0)=z_0$, with $a\in [-\infty, 0)$. Let $p\in \partial \D$ be the starting point of $\eta$. Since $\overline{D}\cap \partial \D=\{\sigma, \tau\}$, it follows that $p\in \{\tau, \sigma\}$, hence, 
 Proposition \ref{Prop:compl-inv-puntos} implies that $z_0$ belongs to the backward invariant set of $(\phi_t)$.
 
 The last statement follows at once from what we already proved and Remark \ref{Rem:backward to complete invariant}.    
\end{proof}

\section{Petals and the geometry of K\"onigs functions}\label{Sec:petals-Koenigs}

In this section we see how  geometric properties of the K\"onigs function of a semigroup  detect petals. To achieve this goal, we use the divergence rate as introduced in \cite{AroBra16}, which, roughly speaking, measures the average hyperbolic speed of escape of an orbit of a semigroup. We recall from \cite{AroBra16} the basic facts we need. 

\begin{definition}
Let $(\phi_t)$ be a continuous one-parameter semigroup of holomorphic self-maps on a Riemann surface $\Omega$. Let $k_\Omega$ denote the hyperbolic distance of $\Omega$. Let $z\in \Omega$. The number 
\[
c_\Omega(\phi_t):=\lim_{s\to+\infty}\frac{k_\Omega(\phi_s(z), z)}{s}
\]
 is called the {\sl divergence rate} of $(\phi_t)$.
\end{definition}
One can prove (see \cite{AroBra16}) that, indeed, the previous limit exists and it is independent of $z\in \Omega$.

\begin{theorem}[\cite{AroBra16}]\label{Thm:speed-dilation}
Let $(\phi_t)$ be a non-elliptic holomorphic semigroup in $\D$ with Denjoy-Wolff point $\tau\in \partial\D$ and spectral value $\lambda\geq 0$. Let $c_\D(\phi_t)$ denote the divergence rate of $(\phi_t)$. Then
\begin{equation}\label{Eq:speed-divergence}
c_\D(\phi_t)=\frac{1}{2} \lambda.
\end{equation}
Moreover, if $(\Omega, h, \Phi_t)$ is a holomorphic model of $(\phi_t)$ then $c_\D(\phi_t)=c_\Omega(\Phi_t)$.
\end{theorem}

 The basic observation we need is contained in the following lemma:

\begin{lemma}\label{Lem:step-petal}
Let $(\phi_t)$ be a semigroup, not a group, in $\D$. Let $(\Omega, h, \psi_t)$ be the canonical model of $(\phi_t)$ (where $\Omega$  and $\psi_t$ are given by Theorem \ref{modelholo}). Suppose $\sigma\in \partial \D$ is a repelling fixed point for $(\phi_t)$ with repelling spectral value $\lambda\in (-\infty,0)$. Let $\Delta$ be the hyperbolic petal of $(\phi_t)$ such that $\sigma\in\partial\Delta$. Let $A:=h(\Delta)$. Then $\psi_t(A)=A$ for all $t\geq 0$ and $(\psi_t|_{A})$ is a continuous group of automorphisms of $A$. Moreover, the divergence rate satisfies
\[
c_\Delta(\phi_t|_\Delta)=c_A(\psi_t|_A)=-\frac{\lambda}{2}.
\]
\end{lemma}
\begin{proof}
Let $(\D, g, \eta_t)$ be a pre-model for $(\phi_t)$ at $\sigma$. By Proposition \ref{Prop:petals-premodel}, $g(\D)=\Delta$. Hence, $(\Delta, g, \phi_t|_\Delta)$ is a holomorphic model for $(\eta_t)$. Therefore, by Theorem \ref{Thm:speed-dilation}, 
\begin{equation}\label{Eq:speed-model-pre1}
c_\Delta(\phi_t|_\Delta)=c_\D(\eta_t)=-\frac{\lambda}{2}.
\end{equation}

Now, let $A=h(\Delta)$. From
\[
A=h(\Delta)=h(\phi_t(\Delta))=\psi_t(h(\Delta))=\psi_t(A), \quad t\geq 0,
\] 
it follows that $(\psi_t|_A)$ is a continuous group of automorphisms of $A$ and that $(\Delta, h|_\Delta, (\phi_t|_\Delta))$ is a holomorphic model for $(\psi_t|_A)$. Hence, again by Theorem \ref{Thm:speed-dilation} and  \eqref{Eq:speed-model-pre1} we have the result. 
\end{proof}

Now, we need to introduce  spirallike sectors.

\begin{definition}
Let $\mu\in \C$, $\Re \mu>0$, $\alpha\in (0,\pi]$ and $\theta_0\in [-\pi,\pi)$. A {\sl $\mu$-spirallike sector}\index{Spirallike sector} of amplitude $2\alpha$ and center $e^{i\theta_0}$ is
\begin{equation*}
\begin{split}
\mathrm{Spir}[\mu, 2\alpha, \theta_0]&:=\{e^{t\mu+i\theta}: t\in \R, \theta\in (-\alpha+\theta_0,\alpha+\theta_0)\}.
\end{split}
\end{equation*}
\end{definition}


\begin{lemma}\label{lem:spiral-sector-rate}
Let $D:=\mathrm{Spir}[\mu, 2\alpha,\theta_0]$ be a $\mu$-spirallike sector for some $\mu\in \C$, $\Re \mu>0$ and $\alpha\in (0,\pi]$. Let $\psi_t(z)=e^{-\mu t}z$, for $z\in \C$. Then $(\psi_t|_{D})$ is a continuous group of automorphisms of $D$ and 
\[
c_{D}(\psi_t|_{D})=\frac{|\mu|^2\pi}{4\alpha\Re \mu }.
\]
\end{lemma}
\begin{proof}
The map $z\mapsto e^{-i\theta_0}z$ is a biholomorphism between $D$ and $\mathrm{Spir}[\mu, 2\alpha,0]$ and conjugate $(\psi_t|_{D})$ to $(\psi_t|_{\mathrm{Spir}[\mu, 2\alpha,0]})$. Since a biholomorphism is an isometry with respect to the hyperbolic distance, the divergence rate of $(\psi_t|_{D})$ and $(\psi_t|_{\mathrm{Spir}[\mu, 2\alpha,0]})$ is the same. Hence, we can assume $\theta_0=0$.

It is clear that $\psi_t(D)=D$ for all $t\geq 0$, hence, $(\psi_t|_D)$ is a continuous group of automorphisms of $D$.
Since $D$ is simply connected and $0\not\in D$, it is well defined a holomorphic  branch of $f: D\ni w\mapsto w^{1-i\frac{\Im \mu}{\Re \mu}}\in \C$. A straightforward computation shows that 
$D':=f(D)=\{\rho e^{i\theta}: \rho>0, \theta\in (-\alpha,\alpha)\}$.

Moreover, $(f\circ \psi_t\circ f^{-1})$ is a continuous group of automorphisms of $D'$. A direct computation shows 
\[
\tilde{\psi}_t(z):=(f\circ \psi_t\circ f^{-1})(z)=e^{-\mu t\left(1-i\frac{\Im \mu}{\Re \mu}\right)}z=e^{-t\frac{|\mu|^2}{\Re \mu}}z, 
\]
 for all $t\geq 0$. Now, consider the function $g:D'\ni z\mapsto z^{-\frac{\pi}{2\alpha}}\in \C$. Then $g(D')=\Ha$ and $(g\circ \tilde{\psi}_t\circ g^{-1})$ is a continuous group of automorphisms of $\Ha$. A direct computation shows that $\eta_t(z):=(g\circ \tilde{\psi}_t\circ g^{-1})(z)=e^{t\frac{\pi |\mu|^2}{2\alpha\Re \mu}}z$ for all $t\geq 0$. That is, $(\eta_t)$ is a group in $\Ha$ which is conjugated to a hyperbolic group $(\tilde\eta_t)$ in $\D$ with spectral value $\frac{\pi |\mu|^2}{2\alpha\Re \mu}$ and then  
 \[
 c_\D(\tilde\eta_t)=\frac{\pi |\mu|^2}{4\alpha\Re \mu}.
 \]
Since $(D, g\circ f, (\psi_t|_D))$ is a holomorphic model for $(\tilde\eta_t)$, it follows $c(\tilde\eta_t)=c_{D}(\psi_t|_{D})$ by  Theorem~\ref{Thm:speed-dilation}. 
\end{proof}

We are now ready to relate petals of a semigroup with the shape of the image of the corresponding  K\"onigs function. As a matter of notation, if $D\subset \C$ is a $\mu$-starlike domain with respect to $0$ for some $\mu\in \C$, $\Re \mu>0$, we say that a $\mu$-spirallike sector $\mathrm{Spir}[\mu, 2\alpha, \theta_0]\subset D$ (for some $\alpha\in [0,\pi)$ and $\theta_0\in [-\pi,\pi)$) is {\sl maximal}\index{Maximal spirallike sector} if there exist no $\theta_1\in [-\pi,\pi), \beta\in (0,2\pi]$ such that 
$\mathrm{Spir}[\mu, 2\alpha, \theta_0]\subset \mathrm{Spir}[\mu, \beta, \theta_1]\subset D$ and  $\mathrm{Spir}[\mu, 2\alpha, \theta_0]\neq\mathrm{Spir}[\mu, \beta, \theta_1]$.

Similarly, if $D\subset \C$ is starlike at infinity, $z_0\in \C$, $\rho>0$, the strip $(\strip_{\rho}+z_0)\subset D$ is {\sl maximal}\index{Maximal strip} if there exist no $r>0$ and $z_1\in \C$ such that $(\strip_{\rho}+z_0)\subset (\strip_{r}+z_1)\subset D$ and $(\strip_{\rho}+z_0)\neq (\strip_{r}+z_1)$.

In the discrete (elliptic, starlike type) case, Theorem \ref{Thm:petals-koenigs} was proven in \cite{Pog98}. For semigroups, 
Theorem \ref{Thm:petals-koenigs} and Theorem \ref{Thm:koenigs-petals} were first proved with a direct, lengthy and more complicated argument in \cite{ConDia05a}. 

\begin{theorem}\label{Thm:petals-koenigs}
Let $(\phi_t)$ be a semigroup, not a group, in $\D$. Let $\tau\in \overline{\D}$ be the Denjoy-Wolff point of $(\phi_t)$ and $\mu$ its spectral value. Let $h$ be the K\"onigs function of $(\phi_t)$. Suppose $\Delta$ is a hyperbolic petal for $(\phi_t)$,  let $\sigma\in \partial\Delta$ be  the unique repelling fixed point for $(\phi_t)$ on $\partial \D$, and let $\lambda\in (-\infty, 0)$ be the repelling spectral value  of $(\phi_t)$ at $\sigma$. 
\begin{enumerate}
\item  If $\tau\in \D$, then there exists $\theta_0\in [-\pi,\pi)$ such that $h(\Delta)$ is a maximal spirallike sector of $h(\D)$ of center $e^{i\theta_0}$ and amplitude $-\frac{|\mu|^2\pi}{\lambda\Re\mu}$, {\sl i.e.},
\[
h(\Delta)=\mathrm{Spir}[\mu, -\frac{|\mu|^2\pi}{\lambda\Re\mu}, \theta_0].
\]
\item If $\tau \in \partial \D$, then there exists $z_0\in \C$ such that $h(\Delta)$ is a maximal strip $z_0+\strip_{-\frac{\pi}{\lambda}}$, {\sl i.e.},
\[
h(\Delta)=z_0+\strip_{-\frac{\pi}{\lambda}}.
\]
\end{enumerate} 
\end{theorem}

\begin{proof}
We first consider the elliptic case. The canonical model is $(\C, h, e^{-\mu t}z)$. Since $\phi_t(\Delta)=\Delta$ for all $t\geq 0$, it follows that $e^{-\mu t}h(\Delta)=h(\Delta)$  and $z\mapsto e^{-\mu t}z$ is an automorphism of $h(\Delta)$ for all $t\geq 0$. In particular, $e^{-\mu t}h(\Delta)=h(\Delta)$ for all $t\in \R$. It follows easily that $h(\Delta)=\mathrm{Spir}[\mu, 2\alpha, \theta_0]$ for some $\theta_0\in [-\pi,\pi)$ and $\alpha\in (0,\pi]$. By Lemma \ref{Lem:step-petal} and Lemma \ref{lem:spiral-sector-rate} it follows at once that $2\alpha=-\frac{|\mu|^2\pi}{\lambda\Re\mu}$. 

It is left to prove that $\mathrm{Spir}[\mu, -\frac{|\mu|^2\pi}{\lambda\Re\mu}, \theta_0]$ is maximal. Suppose this is not the case. Therefore, there exist $\theta_1\in [-\pi,\pi), \beta\in (0,2\pi]$ such that $D:=\mathrm{Spir}[\mu, \beta, \theta_1]\subset h(\D)$ and $\mathrm{Spir}[\mu, -\frac{|\mu|^2\pi}{\lambda\Re\mu}, \theta_0]$ is properly contained in $D$. Therefore, $(e^{-\mu t}z|_D)$ is a continuous group of automorphisms of $D$. Since $D\subset h(\D)$, it follows $\phi_t(h^{-1}(z))=h^{-1}(e^{-\mu t}z)$ for every $z\in h(\D)$, hence $\phi_t(h^{-1}(D))=h^{-1}(D)$ for all $t\geq 0$. This implies that $h^{-1}(D)$ is an open connected component in the backward invariant set of $(\phi_t)$ which properly contains $\Delta$, a contradiction.

The proof in case $(\phi_t)$ is non-elliptic is similar and we  leave the details to the reader.
\end{proof}

The converse of the previous theorem is also true:

\begin{theorem}\label{Thm:koenigs-petals}
Let $(\phi_t)$ be a semigroup, not a group, in $\D$. Let $\tau\in \overline{\D}$ be the Denjoy-Wolff point of $(\phi_t)$ and $\mu$ its spectral value. Let $h$ be the K\"onigs function of $(\phi_t)$. 
\begin{enumerate}
\item  If $\tau\in \D$ and there exist $\theta_0\in [-\pi,\pi)$ and $\beta\in (0,2\pi]$ such that $\mathrm{Spir}[\mu, \beta, \theta_0]\subset h(\D)$ is a maximal spirallike sector of $h(\D)$, then $h^{-1}(\mathrm{Spir}[\mu, \beta, \theta_0])$ is a hyperbolic petal for $(\phi_t)$. Moreover, if  $\sigma\in\partial \Delta$ is the unique repelling fixed point of $(\phi_t)$ contained in $\partial \Delta$, then the repelling spectral value of $(\phi_t)$ at $\sigma$ is $\lambda=-\frac{|\mu|^2\pi}{\beta \Re \mu}$.

\item If $\tau \in \partial \D$ and there exist $z_0\in \C$ and $\rho>0$ such that $z_0+\strip_{\rho}\subset h(\D)$ is a maximal strip of $h(\D)$, then $h^{-1}(z_0+\strip_{\rho})$ is a hyperbolic petal for $(\phi_t)$. Moreover, if  $\sigma\in\partial \Delta$ is the unique repelling fixed point of $(\phi_t)$ contained in $\partial \Delta$, then the repelling spectral value of $(\phi_t)$ at $\sigma$ is $\lambda=-\frac{\pi}{\rho}$.
\end{enumerate} 
\end{theorem}
\begin{proof}
(1) The canonical model is $(\C, h, e^{-\mu t}z)$. Let $D:=\mathrm{Spir}[\mu, \beta, \theta_0]$. Since $\phi_t(h^{-1}(z)=h^{-1}(e^{-\mu t}z)$ for all $z\in h(\D)$ and $t\geq 0$, it follows at once that $\phi_t(D)=D$ for all $t\geq 0$. Hence, there exists a petal $\Delta$ such that $h^{-1}(D)\subseteq \Delta$. However, if $\Delta\neq h^{-1}(D)$, then by Theorem \ref{Thm:petals-koenigs},  $h(\Delta)$ is a spirallike sector properly containing $D$, against the maximality of $D$. Therefore, $\Delta=h^{-1}(D)$ and, the result follows from Theorem \ref{Thm:petals-koenigs}. 

(2) the argument is similar and we omit it.  
\end{proof}

Finally, we turn our attention to parabolic petals. Recall that only parabolic semigroups can have parabolic petals (Remark \ref{Rem:petalo-parab-parab}). 

As a matter of notation, if $W\subset \C$ is a domain starlike at infinity and $a\in \R$, we say that a half-plane $\{w\in \C: \Re w>a\}\subset W$ ({\sl respectively}  $\{w\in \C: \Re w<a\}\subset W$) is {\sl maximal}\index{Maximal half-plane} if $\{w\in \C: \Re w>b\}\not\subset W$ for every $b<a$  ({\sl respect.} $\{w\in \C: \Re w<b\}\not\subset W$ for every $b>a$).

\begin{theorem}\label{Thm:parabolic-petal-geometric-koenigs}
Let $(\phi_t)$ be a parabolic semigroup, not a group, in $\D$, with Denjoy-Wolff point $\tau\in \partial\D$. Let $h$ be the K\"onigs function of $(\phi_t)$. If $\Delta$ is a parabolic petal for $(\phi_t)$ then $h(\Delta)$ is a maximal half-plane in $h(\D)$. Conversely, if $H\subset h(\D)$ is a maximal half-plane in $h(\D)$ then  $h^{-1}(H)$ is a parabolic petal for $(\phi_t)$.

Moreover, $(\phi_t)$ can have at most two parabolic petals and, if this is the case, $(\phi_t)$ has zero hyperbolic step. 
\end{theorem}
\begin{proof}
Let $(\Omega, h, z+it)$ be the canonical model of $(\phi_t)$, where $\Omega=\C, \Ha$ or $\Ha^-$. Let $\Delta$ be a parabolic petal for $(\phi_t)$. Since $\phi_t(\Delta)=\Delta$ for all $t\geq 0$, and $h(\Delta)=h(\phi_t(\Delta))=h(\Delta)+it$ for all $t\geq 0$, it follows that $h(\Delta)+it=h(\Delta)$ for all $t\geq 0$, and, hence,  $h(\Delta)+it=h(\Delta)$ for all $t\in \R$. Therefore, either $h(\Delta)$ is a strip $\strip_\rho+z_0$ for some $\rho>0$ and $z_0\in h(\D)$ or $h(\Delta)$ is a half-plane. Arguing as in the proof of Theorem \ref{Thm:petals-koenigs} it is easy to see that $h(\Delta)$ is maximal in $h(\D)$. Therefore, if $h(\Delta)$ is a maximal strip, the petal $\Delta$ is hyperbolic by Theorem \ref{Thm:koenigs-petals}, contradicting our hypothesis. Hence, $h(\Delta)$ is a maximal half-plane in $h(\D)$. 

Conversely,   if $H\subset h(\D)$ is a maximal half-plane in $h(\D)$, then arguing as in the proof of Theorem \ref{Thm:koenigs-petals}, it follows that $h^{-1}(H)$ is a petal. Moreover, by Theorem \ref{Thm:petals-koenigs}, $h^{-1}(H)$ cannot be hyperbolic, hence, it is parabolic.

It is clear that a domain starlike at infinity (different from $\C$) can contain at most two maximal half-planes, one given by $\{w\in \C: \Re w<a\}$ and the other given by $\{w\in \C: \Re w>b\}$, for some $-\infty<a\leq b<+\infty$. Hence, $(\phi_t)$ can have at most two parabolic petals. 

Finally, assume that $(\phi_t)$ has two parabolic petals $\Delta_1, \Delta_2$. Hence, there exist $-\infty<a\leq b<+\infty$ such that $h(\Delta_1)=\{w\in \C: \Re w<a\}$ and $h(\Delta_2)=\{w\in \C: \Re w>b\}$. This implies that $h(\D)$ is not contained in $\Ha$ or $\Ha^-$, and, by Theorem \ref{modelholo}, it follows that $\Omega=\C$ and $(\phi_t)$ has zero hyperbolic step. 
\end{proof}

\section{Analytic properties of K\"onigs functions at boundary fixed points }

\label{Sec:Boundary Fixed Points and Koenigs Functions}

We start by recalling the following straightforward consequence of \cite[Prop. 3.4, Prop. 3.7]{Gum14}:

\begin{proposition}\label{Prop:fixed-infinity}
Let $(\phi_t)$ be a semigroup, not a group, in $\D$ with Denjoy-Wolff point $\tau\in\overline{\D}$ and  K\"onigs function $h$. Then $\sigma\in \partial \D$ is a fixed point of $(\phi_t)$ if and only if   $\angle\lim_{z\to \sigma}h(z)=\infty\in \C_\infty$. Moreover, if $\sigma\neq \tau$, then (the unrestricted limit) $\lim_{z\to \sigma}h(z)=\infty$.
\end{proposition}

The aim of this section is to  characterize repelling and super-repelling fixed points via K\"onigs functions. In order to properly deal with the elliptic case we need to introduce some terminology. 

Let $\lambda\in \C$, $\Re \lambda>0$. Every point  $w\in \C\setminus\{0\}$ can be written in a unique way in $\lambda$-spirallike coordinates as $w=e^{-\lambda t+i\theta}$, where $t\in \R$ and $\theta\in [-\pi,\pi)$. We define
\[
\Arg_\lambda (w):=\theta,
\]
and we call it the {\sl $\lambda$-spirallike argument of $w$}.

\begin{theorem}\label{Thm:konigs-rep-elliptic}
Let $(\phi_t)$ be an elliptic semigroup in $\D$, not a group, with Denjoy-Wolff point $\tau\in \D$ and spectral value $\lambda\in \C$ with $\Re \lambda>0$. Let $h$ be the associated K\"onigs function and $\sigma\in \partial \D$. The following are equivalent:
\begin{enumerate}
\item $\sigma$ is a repelling fixed point of $(\phi_t)$,
\item $\lim_{z\to \sigma}|h(z)|=\infty$ and $\angle\liminf_{z\to \sigma}\Arg_\lambda (h(z))\neq \angle\limsup_{z\to \sigma}\Arg_\lambda (h(z))$,
\item $\lim_{z\to \sigma}|h(z)|=\infty$ and $\liminf_{z\to \sigma}\Arg_\lambda (h(z))\neq \limsup_{z\to \sigma}\Arg_\lambda (h(z))$.
\end{enumerate}
Moreover, if $\sigma$ is a repelling fixed point for $(\phi_t)$ with repelling spectral value $\nu\in (-\infty,0)$, then there exists $\theta_0\in [-\pi,\pi)$ such that if $\{z_n\}\subset \D$ is a sequence converging to $\sigma$ and $\lim_{n\to \infty}\Arg(1-\overline{\sigma}z_n)=\beta\in (-\pi/2,\pi/2)$, then
\[
\lim_{n\to \infty}\Arg_\lambda (h(z_n))=\theta_0+\frac{\beta |\lambda|^2}{\nu \Re \lambda} \mod [-\pi,\pi).
\]
\end{theorem}
\begin{proof}
(1) implies (2). Suppose $\sigma$ is a repelling fixed point for $(\phi_t)$ with spectral value $\nu\in (-\infty,0)$. By Proposition \ref{Prop:fixed-infinity}, $\lim_{z\to \sigma}|h(z)|=\infty$. By Theorem \ref{Thm:petals-koenigs} there exists a hyperbolic petal $\Delta$ such that $\sigma\in\partial \Delta$ and $h(\Delta)=\mathrm{Spir}[\lambda, 2\alpha, \theta_0]$ is a maximal $\lambda$-spirallike sector of amplitude $2\alpha$ and center $e^{i\theta_0}$ in $h(\D)$, where $\theta_0\in [0,2\pi)$ and $2\alpha=-\pi|\lambda|^2/\nu\Re\lambda$. Let $(\D, g, \eta_t)$ be a pre-model for $(\phi_t)$ at $\sigma$. By Proposition \ref{Prop:petals-premodel}, $g(\D)=\Delta$. Now, take a holomorphic branch of $f:z\mapsto z^{1-i\frac{\Im \lambda}{\Re \lambda}}$ such that the image of $\mathrm{Spir}[\lambda, 2\alpha, \theta_0]$ is $V=\{\rho e^{i\theta}: \rho>0, \theta\in (\theta_0-\alpha, \theta_0+\alpha)\}$. Let then $k:V\to \Ha$ be defined by $k(w)=w^{-\pi/2\alpha}$. By construction, $C:= k  \circ f \circ  h \circ g:\D \to \Ha$ is a biholomorphism. Hence, $C$ is a M\"obius transformation, and, looking at the definition and taking into account that $\lim_{z\to \sigma}|h(z)|=\infty$, we see that $C(\sigma)=0$ and $C(-\sigma)=\infty$. Therefore,
\[
C(z)=\frac{\sigma-z}{\sigma+z}
\]
Now, let $\{z_n\}\subset \D$ be a sequence converging to $\sigma$ such that $\lim_{n\to \infty}\Arg(1-\overline{\sigma}z_n)=\beta\in (-\pi/2,\pi/2)$. By Proposition \ref{Prop:petals-premodel}, $\{z_n\}$ is eventually contained in $g(\D)$ and, without loss of generality, we can assume $\{z_n\}\subset g(\D)$. Let $w_n:=g^{-1}(z_n)$. Since $\angle\lim_{z\to \sigma}g(z)=\sigma$ and $g$ is semi-conformal at $\sigma$, it follows at once that $\lim_{n\to \infty}\Arg(1-\overline{\sigma}w_n)=\beta$. Hence, $\lim_{n\to \infty}\Arg (C(w_n))=\beta$. Therefore, taking into account that $\Arg_\lambda f^{-1}(w)=\Arg (w)$ for all $w\in V$, we have
\[
\lim_{n\to \infty}\Arg_\lambda (h(z_n))=\lim_{n\to \infty}\Arg_\lambda (f^{-1}\circ k^{-1}\circ C\circ g^{-1})(z_n)=\lim_{n\to \infty}\Arg (k^{-1}\circ C)(w_n).
\]
Taking into account that $k^{-1}(z)=e^{i\theta_0}z^{-2\alpha/\pi}$, the previous equation gives immediately
\[
\lim_{n\to \infty}\Arg_\lambda (h(z_n))=\theta_0-\frac{2\alpha\beta}{\pi} \mod [-\pi,\pi).
\]
This proves (2) and, since $2\alpha=-\pi|\lambda|^2/\nu\Re\lambda$,  the final part of the statement.

Clearly, (2) implies (3).

(3) implies (1). Since $\lim_{z\to \sigma}|h(z)|=\infty$ implies $\lim_{z\to \sigma}h(z)=\infty$ in $\C_\infty$, Proposition \ref{Prop:fixed-infinity} immediately implies that $\sigma$ is a boundary fixed point of $(\phi_t)$. We have to show that $\sigma$ is repelling. Let $\underline{x}_\sigma\in \partial_C\D$ be the prime end representing $\sigma$and let $\hat h:\widehat{\D}\to \widehat{h(\D)}$ be the homeomorphism in the Carath\'eodory topology defined by $h$. Since $\lim_{z\to\sigma}h(z)=\infty$, we have $I(\hat{h}(\underline{x}_\sigma))=\{\infty\}$, where $I(\hat{h}(\underline{x}_\sigma))$ is the impression of $\hat{h}(\underline{x}_\sigma)$. Hence, we can find a circular null chain $(C_n)$ representing $\hat{h}(\underline{x}_\sigma)$ such that there exists an increasing sequence of positive real numbers $\{R_n\}$ converging to $+\infty$ such that $C_n\subset \{z\in \C: |z|=R_n\}$ for every $n\in \N_0$. 
For every $n\in \N_0$, let $e^{-\lambda t_n+i\theta^1_n}$ and $e^{-\lambda t_n+i\theta^2_n}$ be the end points of $\overline{C_n}$, where $t_n\in \R$ is such that $e^{-t_n\Re \lambda}=R_n$ and $\theta_n^j\in [-\pi,\pi)$ $j=1,2$ with $\theta^1_n\leq \theta^2_n$.

Given $\mu\in \C$ with $\Re \mu>0$ and $c\in \C\setminus\{0\}$, we let
\begin{equation}\label{spir-def}
\mathrm{spir}_\mu[c]:=\{e^{-\mu s}c:s\in \R\}\cup\{0\}
\end{equation} 

If $\theta^1_n=\theta^2_n$ for all $n\in \N$, then $C_n=\{z\in \C: |z|=R_n\}\setminus\{e^{-\lambda t_n+i\theta^1_n}\}$ for all $n\in \N$. Since $h(\D)$ is $\lambda$-spirallike, it follows that for all $n\in \N$,
\[
h(\D)\cap (\mathrm{spir}_\lambda[e^{-\lambda t_n+i\theta^1_n}]\cap \{w\in \C: |w|\geq R_n\})=\emptyset.
\]
Hence, the only possibility is  that $\theta^1_n$ is constant for all $n$, say, $\theta^1_n=:\theta_0$. That is, $h(\D)=\C\setminus (\mathrm{spir}_\lambda[e^{i\theta_0}]\cap \{w\in \C: |w|\geq R\})$, for some $R\in (0, R_0]$. It is then clear that $\C\setminus \mathrm{spir}_\lambda[e^{i\theta_0}]$ is a maximal $\lambda$-spirallike sector in $h(\D)$ of amplitude $2\pi$.  By Theorem \ref{Thm:koenigs-petals}, there exists a hyperbolic petal $\Delta\subset \D$ such that  $h(\Delta)=\C\setminus \mathrm{spir}_\lambda[e^{i\theta_0}]$. Moreover, if $J:=(\mathrm{spir}_\lambda[e^{i\theta_0}]\cap \{w\in \C: |w|< R\})$, then $\Delta=\D\setminus h^{-1}(J)$.  Note that, by \cite[Lemma 2, p. 162]{Shabook83}, the closure of $h^{-1}(J)$ is a Jordan arc with end points $\tau\in \D$ and a point $p\in \partial \D$. Hence, $\partial \Delta=h^{-1}(J)\cup \partial \D$. By Proposition \ref{Prop:petal-no-other-fixed}, $\partial \Delta$ contains only one boundary fixed point of $(\phi_t)$, which is repelling. Since $\sigma\in \partial \Delta$ is a fixed point, it follows that $\sigma$ is repelling.

Now, we assume that there exists $n_0\in\N$ such that $\theta_{n_0}^1<\theta_{n_0}^2$.  Up to considering the equivalent null chain $(C_n)_{n\geq n_0}$, we can assume $n_0=0$. Since $h(\D)$ is $\lambda$-spirallike, $h(\D)\cap (\mathrm{spir}_\lambda[e^{i\theta_0^j}]\cap \{w\in \C: |w|\geq R_0\})=\emptyset$, $j=1,2$. Hence,  it is easy to see that $\theta_n^1<\theta_n^2$ for all $n\geq 0$.  Let $V_n$ be the interior part of $C_n$, $n\geq 1$.  Since $\mathrm{spir}_\lambda[e^{i\theta_0^1}]\cup \mathrm{spir}_\lambda[e^{i\theta_0^2}]\cup\{\infty\}$, forms a Jordan curve $J$ in the Riemann sphere $\C_\infty$ containing $0$ and $\infty$, taking into account that $h(\D)$ is $\lambda$-spirallike, it follows that $V_n$ is contained in one of the connected component of $\C\setminus J$. Thus, taking also into account that for every $w\in C_n$,  we have $(\mathrm{spir}_\lambda[w]\cap \{w\in \C: |w|<R_n\})\subset h(\D)$, we have two possibilities. Either
\begin{equation}\label{Eq:intra-angle-good}
\begin{split}
&V_n\subseteq (\mathrm{Spir}[\lambda, \theta_n^2-\theta_n^1, \frac{\theta_n^2+\theta_n^1}{2}]\cap \{w\in \C: |w|>R_n\}),\\
&(\mathrm{Spir}[\lambda, \theta_n^2-\theta_n^1, \frac{\theta_n^2+\theta_n^1}{2}]\cap \{w\in \C: |w|<R_n\}) \subset h(\D)
\end{split}
\end{equation}
or, setting $\xi_n:=\frac{\theta_n^2+\theta_n^1+2\pi}{2}\mod 2\pi$, 
\begin{equation}\label{Eq:intra-angle-good-bis}
\begin{split}
&V_n\subseteq (\mathrm{Spir}[\lambda, \theta_n^1-\theta_n^2+2\pi, \xi_n]\cap \{w\in \C: |w|>R_n\}),\\
&(\mathrm{Spir}[\lambda, \theta_n^1-\theta_n^2+2\pi, \xi_n]\cap \{w\in \C: |w|<R_n\}) \subset h(\D).
\end{split}
\end{equation}

If   \eqref{Eq:intra-angle-good} holds for some $n=n_0$, since $V_n\subset V_{n_0}$ for all $n\geq {n_0}$, \eqref{Eq:intra-angle-good} implies that $\theta_n^1<\theta_n^2$ and \eqref{Eq:intra-angle-good} holds for all $n\geq {n_0}$. Assume we are in this case---the proof for the case \eqref{Eq:intra-angle-good-bis} holds for every $n$ is similar and we omit it.

Again, we can assume $n_0=0$. By hypothesis, there exist two sequences $\{z_m\}, \{w_m\}\subset \D$ converging to $\sigma$ such that $\alpha:=\lim_{m\to \infty}\Arg_\lambda (h(z_m))$ and $\beta:=\lim_{m\to \infty}\Arg_\lambda (h(w_m))$ exist and $\alpha<\beta$. Let us write $h(z_m)=e^{-\lambda r_m+i \alpha_m}$ and $h(w_m)=e^{-\lambda s_m+i \beta_m}$, where $r_m, s_m\in \R$ and $\alpha_m, \beta_m\in [-\pi,\pi)$ have the property that $\lim_{m\to \infty}\alpha_m=\alpha$ and $\lim_{m\to \infty}\beta_m=\beta$.  Since $\{z_m\}$ and $\{w_m\}$ are converging to $\sigma$ in the Euclidean topology, it follows  that they also converge to $\underline{x}_\sigma$ in the Carath\'eodory topology of $\D$. Hence, $\{h(z_m)\}$ and $\{h(w_m)\}$ converge to $\hat h(\underline{x}_\sigma)$ in the Carath\'eodory topology of $h(\D)$.It follows that for all $n\in \N$ there exists $m_n\in \N$ such that $h(z_m), h(w_m)\in V_n$ for all $m\geq m_n$. By \eqref{Eq:intra-angle-good}, we have 
$\theta_n^1<\alpha_m, \beta_m<\theta_n^2$ for all $n$ and $m\geq m_{n}$, and, taking the limit in $m$, we obtain
\begin{equation}\label{Eq:infinity-sector}
\theta_n^1\leq \alpha<\beta\leq \theta_n^2, \quad n\in \N.
\end{equation}
Let $\theta_0:=(\beta+\alpha)/2$ and $a\in (0,\beta-\alpha)$. Equation in \eqref{Eq:infinity-sector} implies at once that   $\mathrm{Spir}[\lambda, a, \theta_0]\cap \{w\in \C: |w|>R_n\}\subset V_n$ for all $n\in \N$ and $\mathrm{Spir}[\lambda, a, \theta_0]\subset h(\D)$. By Theorem \ref{Thm:koenigs-petals} there exists a hyperbolic petal $\Delta\subset \D$ such that $\mathrm{Spir}[\lambda, a, \theta_0]\subseteq h(\Delta)$. Moreover, by Proposition \ref{Prop:petal-no-other-fixed}, $\partial \Delta$ contains only one boundary fixed point of $(\phi_t)$ which is repelling. Thus, if we prove that $\sigma\in\partial \Delta$, it follows that $\sigma$ is repelling. To this aim, consider the curve $\gamma:(-\infty, 0) \ni t\mapsto e^{-t\lambda + i\theta_0}$. Since for all $n\in \N$ there exists $t_n\in (-\infty,0)$ such that $\gamma(t)\in \mathrm{Spir}[\lambda, a, \theta_0]\cap \{w\in \C: |w|>R_n\}\subset V_n$ for all $t\leq t_n$, it follows  that $\gamma(t)$ converges in the Carath\'eodory topology of $h(\D)$ to $\hat h(\underline{x}_\sigma)$ as $t\to -\infty$. Hence,  $h^{-1}(\gamma(t))\to \sigma$ as $t\to -\infty$, proving that $\sigma\in \partial \Delta$.
 \end{proof}

As an immediate corollary of Proposition \ref{Prop:fixed-infinity} and Theorem \ref{Thm:konigs-rep-elliptic} we have:

\begin{corollary}\label{Cor:fijo-super-elliptico}
Let $(\phi_t)$ be an elliptic semigroup in $\D$, not a group. Let $h$ be the associated K\"onigs function and $\sigma\in \partial \D$. The following are equivalent:
\begin{enumerate}
\item $\sigma$ is a super-repelling fixed point of $(\phi_t)$,
\item $\lim_{z\to \sigma}|h(z)|=\infty$ and $\lim_{z\to \sigma}\Arg_\lambda (h(z))=\theta_0$ for some  $\theta_0\in [-\pi,\pi]$.
\end{enumerate}
\end{corollary}

Now we turn our attention to non-elliptic semigroups. The proofs of the next results are similar to those for the elliptical case and we omit them---in fact, 
roughly speaking,  in the non-elliptic case, the role of the modulus is played by the imaginary part and that of the $\lambda$-argument by the real part.  

\begin{theorem}\label{Thm:konigs-rep-nonelliptic}
Let $(\phi_t)$ be a  non-elliptic semigroup in $\D$, not a group, with Denjoy-Wolff point $\tau\in \partial \D$ and spectral value $\lambda\geq 0$. Let $h$ be the associated K\"onigs function and $\sigma\in \partial \D$. The following are equivalent:
\begin{enumerate}
\item $\sigma$ is a repelling fixed point of $(\phi_t)$,
\item $\lim_{z\to \sigma}\Im h(z)=-\infty$ and $-\infty<\angle\liminf_{z\to \sigma}\Re h(z)\neq \angle\limsup_{z\to \sigma}\Re h(z)<+\infty$,
\item $\lim_{z\to \sigma}\Im h(z)=-\infty$ and $-\infty<\liminf_{z\to \sigma}\Re h(z)\neq \limsup_{z\to \sigma}\Re h(z)<+\infty$.
\end{enumerate}
Moreover, if $\sigma$ is a repelling fixed point for $(\phi_t)$ with repelling spectral value $\nu\in (-\infty,0)$, then there exists $a\in \R$ such that if $\{z_n\}\subset \D$ is a sequence converging to $\sigma$ and $\lim_{n\to \infty}\Arg(1-\overline{\sigma}z_n)=\beta\in (-\pi/2,\pi/2)$, then
\[
\lim_{n\to \infty}\Re h(z_n)=a+\frac{\beta}{\nu}.
\]
\end{theorem}

As a corollary  we have:

\begin{corollary}\label{Cor:fijo-super-nonelliptico}
Let $(\phi_t)$ be a non-elliptic semigroup in $\D$, not a group. Let $h$ be the associated K\"onigs function and $\sigma\in \partial \D$. The following are equivalent:
\begin{enumerate}
\item $\sigma$ is a super-repelling fixed point of $(\phi_t)$,
\item $\lim_{z\to \sigma}\Im h(z)=-\infty$ and $\lim_{z\to \sigma}\Re h(z)=a$ for some  $a\in \R$.
\end{enumerate}
\end{corollary}

The previous results allow also to easily prove the following:

\begin{corollary}
Let $(\phi_t)$ be a semigroup, not an elliptic group, in $\D$ with Denjoy-Wolff point $\tau\in\overline{\D}$ and let $h$ be its K\"onigs function. Let $\sigma\in\partial \D$. Then $\sigma$ is not a fixed point of $(\phi_t)$ if and only if $\angle\lim_{z\to \sigma}h(z)\in \C$. Moreover, in this case,
\begin{enumerate}
\item if $\tau\in \D$ then $\lim_{z\to \sigma}\Arg_\lambda h(z)$ exists, where $\lambda$ is the spectral value of $(\phi_t)$;
\item if $\tau\in\partial\D$ then $\lim_{z\to \sigma}\Re h(z)$ exists.
\end{enumerate}
\end{corollary}

\section{Examples}
\label{Sec:examples}

Let $(\phi_t)$ be a semigroup, not a group, in $\D$ and let $h$ be its K\"onigs function. Recall that, by Theorem \ref{Thm:petals-koenigs} and Theorem \ref{Thm:koenigs-petals} there is a one-to-one correspondence between hyperbolic petals of $(\phi_t)$ and maximal strips in the non-elliptic case (or maximal spirallike sectors in the elliptic case) in $h(\D)$. Moreover, the repelling spectral value can be read by the width of the strip (or the angle of the spirallike sector). Also,  by Theorem \ref{Thm:parabolic-petal-geometric-koenigs}, there is a one-to-one correspondence between parabolic petals and maximal half-planes. The previous developed theory allows also to read information on the boundary of a petal using directly the image of $h$. We summarize and translate here the results in a suitable manageable way. We start with the elliptic case (recall Definition \ref{spir-def}):

\begin{proposition}\label{Prop:leer-frontera-h-petalo-ell}
Let $(\phi_t)$ be an elliptic semigroup, not a group, with Denjoy-Wolff point $\tau\in \D$ and spectral value $\lambda\in \C$, $\Re\lambda>0$ and let $h$ be its K\"onigs function. Let $\Delta$ be a hyperbolic petal which corresponds to the maximal spirallike sector $\mathrm{Spir}[\mu, 2\alpha, \theta_0]$, for some $\alpha\in [0,\pi)$ and $\theta_0\in [-\pi, \pi)$. Let $\sigma\in \partial \D\cap \partial \Delta$ be the only repelling fixed point of $(\phi_t)$ contained in $\partial\Delta$. Let $S:=\mathrm{spir}_\lambda[e^{i(\theta_0+\alpha})]\setminus\{0\}$ or $S:=\mathrm{spir}_\lambda[e^{i(\theta_0-\alpha})]\setminus\{0\}$. Then one and only one of the following happens:
\begin{enumerate}
\item There exists $a>0$ such that $S\cap \{w\in \C: |w|<a\}\subset h(\D)$ and $S\cap \{w\in \C: |w|\geq a\}\cap h(\D)=\emptyset$. This is the case if and only if $h^{-1}(S\cap \{w\in \C: |w|<a\})$ is a connected component of $\partial \Delta\cap (\D\setminus\{\tau\})$ whose closure is a Jordan arc with end points $\tau$ and a non-fixed point $p\in \partial \D$ such that $\angle\lim_{z\to p}h(z)=S\cap \{w\in \C: |w|=a\}$. 
\item $S\subset h(\D)$. This is the case if and only if $h^{-1}(S)$ is a connected component of $\partial \Delta\cap (\D\setminus\{\tau\})$ whose closure is a Jordan arc with end points $\tau$ and $\sigma$. 
\end{enumerate}
\end{proposition} 
\begin{proof} Assume $S:=\mathrm{spir}_\lambda[e^{i(\theta_0+\alpha})]\setminus\{0\}$ (the other case is similar).

(1) Clearly, $h^{-1}(S\cap \{w\in \C: |w|<a\})$ is a connected component of $\partial \Delta\cap (\D\setminus\{\tau\})$. By  Corollary \ref{Cor:frontera-interior-petalo}, the closure of $h^{-1}(S\cap \{w\in \C: |w|<a\})$ is a Jordan arc joining $\tau$ with a point $p\in\partial\D$ which can be either $\sigma$ or a non-fixed point. Let $\gamma:(-\infty, t_0)\ni t\mapsto e^{\lambda t+i(\theta_0+\alpha)}$ be a parameterization of $S$, with $t_0\in \R$ such that $e^{\lambda t_0+i(\theta_0+\alpha)}=S\cap \{w\in \C: |w|=a\}$. Then  $\lim_{t\to t_0}h^{-1}(\gamma(t))=p$. Since $\lim_{t\to t_0}h(h^{-1}(\gamma(t)))=S\cap \{w\in \C: |w|=a\}$, Lehto-Virtanen Theorem, implies that $\angle\lim_{z\to p}h(z)=S\cap \{w\in \C: |w|=a\}$. In particular, $\angle\lim_{z\to p}|h(z)|<+\infty$, and hence $p$ is not a fixed point by Proposition \ref{Prop:fixed-infinity}.

(2) The argument is similar and we leave the proof to the reader. 
\end{proof}

A similar argument allows to handle the non-elliptic case:

\begin{proposition}\label{Prop:leer-frontera-h-petalo-nell}
Let $(\phi_t)$ be a  non-elliptic semigroup, not a group, with Denjoy-Wolff point $\tau\in \partial\D$  and let $h$ be its K\"onigs function. Let $\Delta$ be a hyperbolic petal which corresponds to the maximal strip $S=\{w\in \C: a_{1}<\Re w<a_{2}\}$ for some $a_{1}, a_{2}\in \R$, $a_{1}<a_{2}$. Let $\sigma\in \partial \D\cap \partial \Delta$ be the only repelling fixed point of $(\phi_t)$ contained in $\overline{\Delta}$. Fix $j\in \{1,2\}$.
 Then one and only one of the following happens for $j=1, 2$:
\begin{enumerate}
\item $\{w\in \C: \Re w=a_{j}\}\cap h(\D)=\emptyset$.
\item There exists $r\in \R$ such that $\{w\in \C: \Re w=a_{j}, \Im w>r\}\subset h(\D)$ and $\{w\in \C: \Re w=a_{j}, \Im w\leq r\}\cap h(\D)=\emptyset$. This is the case if and only if $h^{-1}(\{w\in \C: \Re w=a_{j}, \Im w>r\})$ is a connected component of $\partial \Delta\cap\D$ whose closure is a Jordan arc with end points $\tau$ and a non-fixed point $p\in \partial \D$ such that $\angle\lim_{z\to p}h(z)=a_{j}+ir$. 
\item $\{w\in \C: \Re w=a_{j}\}\subset h(\D)$. This is the case if and only if $h^{-1}(\{w\in \C: \Re w=a_{j}\})$ is a connected component of $\partial \Delta\cap \D$ whose closure is a Jordan arc with end points $\tau$ and $\sigma$. 
\end{enumerate}
\end{proposition} 

Finally, we have the parabolic petals:

\begin{proposition}\label{Prop:leer-frontera-p-petalo}
Let $(\phi_t)$ be a  parabolic semigroup, not a group, with Denjoy-Wolff point $\tau\in \partial\D$  and let $h$ be its K\"onigs function. Let $\Delta$ be a parabolic petal which corresponds to the maximal half-plane $H=\{w\in \C: \Re w>a\}$ for some $a\in \R$. 
 Then one and only one of the following happens:
\begin{enumerate}
\item There exists $r\in \R$ such that $\{w\in \C: \Re w=a, \Im w>r\}\subset h(\D)$ and $\{w\in \C: \Re w=a, \Im w\leq r\}\cap h(\D)=\emptyset$. This is the case if and only if $h^{-1}(\{w\in \C: \Re w=a, \Im w>r\})$ is a connected component of $\partial \Delta\cap \D$ whose closure is a Jordan arc with end points $\tau$ and a non-fixed point $p\in \partial \D$ such that $\angle\lim_{z\to p}h(z)=a+ir$. 
\item $\{w\in \C: \Re w=a\}\subset h(\D)$. This is the case if and only if $h^{-1}(\{w\in \C: \Re w=a\})$ is a connected component of $\partial \Delta\cap \D$ whose closure $J$ is a Jordan curve with  $J\cap \partial \D=\{\tau\}$. 
\end{enumerate}
\end{proposition} 
\begin{proof}
The proof is similar to that of Proposition \ref{Prop:leer-frontera-h-petalo-ell}, so we leave it to the reader. The only issue here is to show that the case $\{w\in \C: \Re w=a\}\cap h(\D)=\emptyset$ cannot happen. Indeed, if this is the case, then $h(\D)=H$ and $(\phi_t)$ is a parabolic group. 
 \end{proof}

A similar result holds in case the maximal half-plane associated with the parabolic petal is $H=\{w\in \C: \Re w<a\}$ for some $a\in \R$.

\smallskip

We give now a list of examples of petals of all types described in Proposition \ref{Prop:shape-of-petals}.

\begin{example} \label{Ex:petal type1}
Let $h$ be the Koebe function $h(z):=\frac{z}{(1-z)^{2}}$, $z\in \D$. It is univalent and $h(\D)=\C\setminus (-\infty, -1/4]$. Consider the semigroup whose model is $(\C,h,e^{-t}z)$, that is, $\phi_{t}(z):=h^{-1}(e^{-t}h(z))$, for all $z\in \D$ and $t\geq 0$. Since $\cap_{t\geq 0}e^{-t}h(\D)\setminus \{0\}=\C\setminus (-\infty, 0]=\mathrm{Spir}[-1, -\pi, 2\pi]$ and it is a maximal spirallike sector of $h(\D)$, Theorem \ref{Thm:koenigs-petals} shows that 
\[
h^{-1}(\mathrm{Spir}[-1, -\pi, 2\pi])=\D\setminus (-1,0]
\]
 is a hyperbolic petal for $(\phi_{t})$. Clearly, it is the unique petal of the semigroup. Therefore, $(\phi_t)$ has a unique boundary fixed point $\sigma\in \partial \D$, which is repelling with repelling spectral value $-2\pi$.  Since $\lim_{(0,1)\ni r\to 1}h(r)=\infty$, we get $\sigma=1$ by  Proposition \ref{Prop:petal-no-other-fixed}. This petal is an example of  the type described in Proposition \ref{Prop:shape-of-petals}(1).
\end{example}

\begin{example} \label{Ex:petal type2}
Consider the domain
$$
\Omega=\strip \cup (1+i\R)\cup \{w\in \strip+1: \Im w (1-\Re w)>1\}\subset\strip_2.
$$
Let $h:\D \to \C$ be univalent  such that  $h(\D)=\Omega $. By construction, $\Omega+it\subset \Omega $ for all $t\geq 0$ and $\cup_{t\geq 0}(\Omega-it)=\strip_{2}$. 
Consider the semigroup whose model is $(\strip_2,h,z+it)$, that is, $\phi_{t}(z):=h^{-1}(h(z)+it)$, for all $z\in \D$ and $t\geq 0$. Let us call $\tau\in \partial \D$ its Denjoy-Wolff point. 
Since $\cap_{t\geq 0} (\Omega+it)=\cup _{x\in (0,1]}(x+i\R)=\strip\cup (1+i\R)$, $h(\D)$ contains a unique maximal strip,  $\strip$, whose boundary is $i\R\cup (1+i\R)$. Let $\Delta=h^{-1} (\strip)$. 
Therefore, $\Delta$ is the unique hyperbolic petal of the semigroup $(\phi_t)$. Let us denote by $\sigma$ the repelling fixed point associated with $\Delta$ given by Proposition \ref{Prop:petal-no-other-fixed}.  By Proposition \ref{Prop:leer-frontera-h-petalo-nell}, the maximal invariant curve $\R\ni t\mapsto h^{-1}(1+it)$
is a connected component of $\partial \Delta\cap \D$ whose closure is a Jordan arc with end points $\tau$ and $\sigma$.  It divides the unit disc in two connected components, one of them is the petal, and the other one is $B=h^{-1} (\{w\in \strip+1: \Im w (1-\Re w)>1\})$. Clearly, $\overline \Delta \cap \partial \D$ and $\overline B\cap \partial \D $ are the two Jordan arcs in $\partial \D$ that joins $\sigma$ and $\tau$. Let us denote by $J$ the one which  is included in $\partial \Delta$. 
Then $\partial \Delta=\{\tau, \sigma\}\cup h^{-1}(1+i\R)\cup J$. Thus $\Delta$ is an example of a petal of the type described in  Proposition \ref{Prop:shape-of-petals}(2). 
\end{example}

\begin{example} \label{Ex:petal type2-2}
Consider the domain
\begin{equation*}
\Omega=\{w\in \strip: \Im w >0\}\cup (1+i(0,+\infty))\cup  (\Ha+1)\subset \Ha.
\end{equation*}
Let $h:\D \to \C$ be univalent  such that  $h(\D)=\Omega $. Since $\overline{\Omega}^\infty$ is a Jordan domain, there exists a homeomorphism $\tilde{h}:\overline{\D}\to \overline{\Omega}^\infty$ such that $\tilde{h}|_\D=h$.
Let $(\phi_t)$ be the parabolic semigroup  whose model is $(\Ha,h,z+it)$, that is, $\phi_{t}(z):=h^{-1}(h(z)+it)$, for all $z\in \D$ and $t\geq 0$. Let $\tau\in \partial \D$ be its Denjoy-Wolff point. 
Note that $\cap_{t\geq 0} (\Omega+it)=\cup _{x\in (1,+\infty)}(x+i\R)=\Ha+1$. Let  $\Delta=h^{-1} (\Ha+1)$. 
Therefore, $\Delta$ is the unique (parabolic) petal of the semigroup $(\phi_t)$. Write $J=h^{-1}(1+i(0,+\infty))\subset \D$, $\sigma =\tilde h^{-1}(1)$ and $A=\tilde h^{-1}(1+i(-\infty,0))\subset \partial  \D$. Then $\partial \Delta=\{\tau\}\cup J\cup \{\sigma\}\cup A$. Thus $\Delta$ is an example of a petal described in  Proposition \ref{Prop:shape-of-petals}(2). The difference with Example \ref{Ex:petal type2} is that $\sigma$ in this case is not a boundary fixed point. 
\end{example}

\begin{example} \label{Ex:petal type3}
Consider the domain
\begin{equation*}
\begin{split}
\Omega=&\{w\in \strip: \Im w (\Re w-1)<\Re w\}\cup (1+i\R)\cup (\strip+1)\cup\\  
&\qquad\qquad\cup(2+i\R)\cup\{w\in \strip+2: \Im w (\Re w-2)>\Re w-3\}\subset\strip_3.
\end{split}
\end{equation*}
Let $h:\D \to \C$ be univalent  such that  $h(\D)=\Omega $. By construction, $\Omega+it\subset \Omega $ for all $t\geq 0$ and $\cup_{t\geq 0}(\Omega-it)=\strip_{3}$. 
Consider the semigroup whose model is $(\strip_3,h,z+it)$, that is, $\phi_{t}(z):=h^{-1}(h(z)+it)$, for all $z\in \D$ and $t\geq 0$. Let  $\tau\in \partial \D$ be its Denjoy-Wolff point. 
Since $\cap_{t\geq 0} (\Omega+it)=\cup _{x\in [1,2]}(x+i\R)$, its interior has a unique maximal strip whose boundary is $(1+i\R)\cup (2+i\R)$. Let $\Delta=h^{-1} (\cup _{x\in (1,2)}(x+i\R))=h^{-1}(\strip+1)$. 
Therefore, $\Delta$ is the unique (hyperbolic) petal of the semigroup $(\phi_t)$. Let us denote by $\sigma$ the repelling fixed point associated with $\Delta$ given by Proposition \ref{Prop:petal-no-other-fixed}. By Proposition \ref{Prop:leer-frontera-h-petalo-nell}(3), $h^{-1}(1+i\R)$ and $h^{-1}(2+i\R)$ are connected components of $\partial \Delta\cap \D$ whose closures  are Jordan arcs with end points $\tau$ and $\sigma$. According to Proposition \ref{Prop:shape-of-petals},  $\partial \Delta=\{\tau, \sigma\}\cup h^{-1}(1+i\R)\cup h^{-1}(2+i\R)$. Thus $\Delta$ is an example of a petal  described in  Proposition \ref{Prop:shape-of-petals}(3).
\end{example}

\begin{example} \label{Ex:petal type4}
Consider the domain
\begin{equation*}
\begin{split}
\Omega=&\{w\in \strip: \Im w >0\}\cup (1+i(0,+\infty))\cup (\strip+1)\cup\\  
&\qquad \qquad\cup(2+i(0,+\infty))\cup\{w\in \strip+2: \Im w>0\}\subset\strip_3.
\end{split}
\end{equation*}
Let $h:\D \to \C$ be univalent  such that  $h(\D)=\Omega $.
Consider the semigroup whose model is $(\strip_3,h,z+it)$, that is, $\phi_{t}(z):=h^{-1}(h(z)+it)$, for all $z\in \D$ and $t\geq 0$. Denote by  $\tau\in \partial \D$ its Denjoy-Wolff point. 
Note that $\cap_{t\geq 0} (\Omega+it)=\cup _{x\in (1,2)}(x+i\R)=\strip+1$. Let  $\Delta=h^{-1} (\strip+1)$. 
Therefore, $\Delta$ is the unique hyperbolic petal of the semigroup $(\phi_t)$. Let us denote by $\sigma$ the repelling fixed point associated with $\Delta$ given by Proposition \ref{Prop:petal-no-other-fixed}. Write $J_j=h^{-1}(j+i(0,+\infty))\subset \D$, $j=1,2$. 
Proposition \ref{Prop:leer-frontera-h-petalo-nell}(2) guarantees that $J_{1}$ and $J_{2}$ are connected components of $\partial \Delta\cap \D$ with end points $\tau$ and a non-fixed point $p_{j}\in \partial \D$ such that $\angle\lim_{z\to p_{j}}h(z)=j+ir$. By default, looking at Proposition \ref{Prop:shape-of-petals}, $\Delta$ is an example of a  described in  statement  (4). Therefore, if $A\subset \partial \D$ is the  closed arc with end points $p_1$ and $p_2$ which does not pass through $\tau$, we have that 
$\partial \Delta=\{\tau\}\cup J_1\cup J_2 \cup A$. 
\end{example}

\begin{example} \label{Ex:petal type5}
Consider the domain
\begin{equation*}
\Omega=\{w\in \strip: \Im w (\Re w-1)<\Re w\}\cup \overline{ (\Ha+1)}\subset \Ha.
\end{equation*}
Let $h:\D \to \C$ be univalent  such that  $h(\D)=\Omega $.
Consider the parabolic semigroup whose model is $(\Ha,h,z+it)$, that is, $\phi_{t}(z):=h^{-1}(h(z)+it)$, for all $z\in \D$ and $t\geq 0$. Let  $\tau\in \partial \D$ be its Denjoy-Wolff point. 
Note that $\cap_{t\geq 0} (\Omega+it)=\cup _{x\in [1,+\infty)}(x+i\R)=\overline {\Ha+1}$ and write $\Delta=h^{-1} (\Ha+1)$. 
Therefore, $\Delta$ is the unique parabolic petal of the semigroup $(\phi_t)$. Let $A=h^{-1}(1+i\R)\subset \D$. Then Proposition \ref{Prop:leer-frontera-p-petalo}(2) shows that $A$
is a connected component of $\partial \Delta\cap \D$ and $\partial \Delta=\{\tau\}\cup A$. Thus $\Delta$ is an example of a petal described in  Proposition \ref{Prop:shape-of-petals}(5).
\end{example}

\section{An example of a semigroup with a non-regular pre-model}\label{sec-nr}

We end up this paper by constructing an example of a semigroup with a repelling fixed point $\sigma$, a pre-model $(\D, g, \eta_t)$ for $(\phi_t)$ at $\sigma$ such that $g$ is {\sl not regular} at $\sigma$.  

Some remarks are in order. Let  $(\D, g, \eta_t)$  be a pre-model for $(\phi_t)$ at a repelling fixed point $\sigma\in\partial \D$. Let $\Delta=g(\D)\subset \D$ be the associated petal.

Since by (2) in the definition of pre-model $\angle\lim_{z\to \sigma}g(z)=\sigma$, by definition, $g$ is {\sl regular} at $\sigma$ if $\angle\lim_{z\to \sigma}\frac{\sigma-g(z)}{\sigma-z}$ exists (finite). 
 By the Julia-Wolff-Carath\'eodory theorem, this condition  is equivalent to 
$\alpha:=\liminf_{z\to \sigma}\frac{1-|g(z)|}{1-|z|}<+\infty$ and, in fact,  $\angle\lim_{z\to \sigma}\frac{\sigma-g(z)}{\sigma-z}=\alpha$. Since $\eta_{-t}(z)$ converges to $\sigma$ non-tangentially  as $t\to+\infty$ for every $z\in \D$, the previous argument implies that $g$ is regular at $\sigma$ if and only if $\limsup_{t\to +\infty}\frac{1-|g(\eta_{-t}(z))|}{1-|\eta_{-t}(z)|}<+\infty$ for every $z\in\D$, 
which,  in turn, is equivalent to
\begin{equation}\label{reg-dw}
\limsup_{t\to +\infty}[\omega(0,\eta_{-t}(z))-\omega(0,g(\eta_{-t}(z)))]<+\infty, \quad \forall z\in \D.
\end{equation}
The map $g$ is a biholomorphism on the image---hence an isometry for the hyperbolic distance---and $g\circ \eta_{-t}=(\phi_t|_\Delta)^{-1}\circ g$ for all $t\geq 0$. Therefore,  
\[
\omega(0,\eta_{-t}(z))=k_\Delta(g(0),g(\eta_{-t}(z)))=k_\Delta(g(0), (\phi_t|_\Delta)^{-1}(g(z))). 
\]
If the semigroup $(\phi_t)$ is parabolic, with universal model $(\C, h, z+it)$, then by Theorem \ref{Thm:petals-koenigs}(2), $h(\Delta)$ is a maximal strip $S\subset h(\D)$. Therefore, taking into account that $h\circ (\phi_t|_\Delta)^{-1}(g(z))=h(g(z))-it$ for all $t\geq 0$, all $z \in \D$, and that $g$ is regular at $\sigma$, it turns out that \eqref{reg-dw} is equivalent to
\begin{equation*}
\limsup_{t\to +\infty}[k_S(h(g(0)),h(g(z))-it)-k_{h(\D)}(h(0),h(g(z))-it)]<+\infty, \quad \forall z\in \D.
\end{equation*}
Using the triangle inequality, setting $w:=h(g(z))$, this latter inequality is equivalent to
\begin{equation}\label{reg-dw2}
\limsup_{t\to +\infty}[k_S(w,w-it)-k_{h(\D)}(w,w-it)]<+\infty \quad \forall w\in S.
\end{equation}

We construct our example as follows. Let $\{y_k\}$ be a strictly decreasing sequence of real numbers,  to be suitable chosen later, such that  $\lim_{k\to \infty}y_k=-\infty$. For $k=1,2,\ldots$, let
\[
D_k:=\C\setminus \{z\in \C: \Re z=\pm (1+\frac{1}{k}), \Im z\leq y_k\}.
\] 
Let $D:=\bigcap_k D_k$. Note that $D$ is a simply connected domain, $D+it\subset D$ for all $t\geq 0$, and $D$ is symmetric with respect to the imaginary axis. Moreover, $D$ contains a maximal strip $S:=\{z\in \C: -1<\Re z<1\}$. 

Let $h:\D\to D$ be a Riemann map. Setting $\phi_t(z):=h^{-1}(h(z)+it))$, $z\in \D$, $t\geq 0$, it follows that $(\phi_t)$ is a holomorphic semigroup of $\D$ with universal model $(\C, h, z+it)$.  By Theorem \ref{Thm:koenigs-petals}, the maximal strip $S$ in $D$ corresponds to a hyperbolic petal, whose closure contains a unique repelling point of $(\phi_t)$, say $\sigma\in \partial \D$. Let $(\D, g, \eta_t)$ be a pre-model for $(\phi_t)$ at $\sigma$.  According to \eqref{reg-dw2}, $g$ is not regular at $\sigma$ if 
\begin{equation}\label{no-reg3}
\limsup_{t\to +\infty}[k_S(0,-it)-k_{D}(0,-it)]=+\infty. 
\end{equation}

Recall that, if $r>0$ and $S_{2r}=\{w\in \C:\, -r<\Re w<r\}$,  then for all $s,t\in \R$, 
\[
k_{S_{2r}}(is,it)=\frac{\pi}{4r}|t-s|.
\]

Given $p\in \D$ and $R>0$, we denote by $D^{hyp}(p,R)$ the hyperbolic disc of center $p$ and  radius $R$, that is, $D^{hyp}(p,R):=\{z\in \D:\, \omega(z,p)<R\}$. Since $\cup_{R>0}D^{hyp}(p,R)=\D$,  it follows that for every $z,w\in \D$, $\lim_{R\to +\infty} k_{D^{hyp}(p,R)}(z,w)=\omega(z,w)$. From this observation, it is not hard to prove the following lemma:

\begin{lemma}\label{Lem:hyp dist in hyp disc}
For all $c>1$ and $M>0$, there exists $R>0$ such that for all $p\in \D$,  
$$
k_{D^{hyp}(p,R)}(z,w)\leq c \, \omega (z,w)
$$
for all $z,w\in D^{hyp}(p,M)$.
\end{lemma} 

Now we need a localization lemma which allows to suitably choose the sequence $\{y_n\}$:

\begin{lemma}\label{Lem:pre model no conforme}
There exist two strictly decreasing sequences $\{y_{k}\}$ and $\{\alpha_k\}$ of real numbers, both converging to $-\infty$, such that, for each $k\in \N$, 
\[
k_D((\alpha_k+\frac{1}{2})i, (\alpha_k-\frac{1}{2})i)\leq (1+\frac{1}{2k})k_{D_k}((\alpha_k+\frac{1}{2})i, (\alpha_k-\frac{1}{2})i). 
\]
\end{lemma}
\begin{proof} Given $p\in D_k$ and $R>0$, let $B_k(p,R):=\{ w\in D_k: \, k_{D_k} (p,w)<R\}$. Let $S_{k}:=\{w\in \C:\, |\Re z|<1+1/k\}$. Clearly $S_{k}\subset D_{k}$ for all $k$. Hence, 
\begin{equation}\label{Eq:pre model no conforme2}
k_{D_{k}}((\alpha_{k}+1/2)i,\alpha_{k}i)\leq k_{S_{k}}((\alpha_{k}+1/2)i,\alpha_{k}i)=\frac{\pi}{8(1+1/k)}<\pi. 
\end{equation}

Set $y_{1}=0$ and $\alpha_{1}=-1/2$. By Lemma \ref{Lem:hyp dist in hyp disc} and the invariance of hyperbolic distance with respect to biholomorphisms, there exists $R_{1}>0$ such that for all $z,w\in B_{1}(-i/2,\pi)$, we have 
$k_{B_{1}(-i/2,R_{1})}(z,w)\leq \left(1+\frac{1}{2}\right) k_{D_{1}}(z,w)$. In particular, by \eqref{Eq:pre model no conforme2},
\begin{equation}\label{Eq:pre model no conforme3}
k_{B_{1}(-i/2,R_{1})}(0,-\frac{i}{2})\leq \left(1+\frac{1}{2}\right) k_{D_{1}}(0,-\frac{i}{2}).
\end{equation}
Since the closure of $B_{1}(-i/2,R_{1})$ is compact in $D_{1}$, $
\beta_{1}:=\inf\{\Im w:\, w\in B_{1}(-i/2,R_{1})\}>-\infty$.

If $y_{j}<\beta_{1}$ for all $j\geq 2$, then  $ B_{1}(-i/2,R_{1})\subset D$. Hence, $k_{D}(0,-\frac{i}{2})\leq k_{B_{1}(-i/2,R_{1})}(0,-\frac{i}{2})$ and the statement holds for $k=1$ by \eqref{Eq:pre model no conforme3}. 

Now we choose $y_2$ and $\alpha_2$. Let $y_{2}<\beta_{1}$ and let $f_{2}:\D\to D_2$ be a Riemann map. By \cite[Lemma 2, p. 162]{Shabook83}, there exists $p:=\lim_{t\to-\infty} f_{2}^{-1}(it)$. Let $\Gamma:=   f_{2}^{-1}(D_2\setminus D_1)$.

For each $n=0,1,\ldots$, let $J_n$ be the segment in $D_2$ with extreme points $-1-1/2+i(y_2-n)$ and $1+1/2+i(y_2-n)$ and $C_n=f_2^{-1}(J_n)$. Then $(C_n)$ is a null chain in $\D$ and its impression is the singleton $p$. Clearly,  $\overline{\Gamma}$ does not intersect the interior part of $C_n$ for all $n$ and then $p\not\in \overline{\Gamma}$.

Let $R_{2}$ be given by Lemma \ref{Lem:hyp dist in hyp disc} with $M=\pi$ and $c=1+1/4$. Since $p\notin \overline{\Gamma}$, there exists $\alpha_{2}<y_{2}$ such that
$D^{hyp}(f_{2}^{-1}(\alpha_{2}i),R_{2})\cap \overline{\Gamma}=\emptyset$. Hence, $B_{2}(\alpha_{2}i,R_{2})\subset D_{1}\cap D_{2}$. Therefore, arguing as before, the statement holds for $k=2$.

Let $\beta_{2}=\inf\{\Im w:\, w\in B_{2}(\alpha_{2}i,R_{2})\}>-\infty.$ Since $\alpha_{2}<y_{2}<\beta_{1}$, we get that $\beta_{2}<\beta_{1}$ and the lemma is proved repeating the previous argument by induction.
 \end{proof}
 
Now, choose the sequences $\{y_k\}, \{\alpha_k\}$ as in the previous lemma. Taking into account that $D$ is symmetric with respect to the imaginary axis, hence the imaginary axis is a geodesic for the hyperbolic distance, we have 
\begin{equation*}
\begin{split}
&k_{S}(0, i(\alpha_k+1/2))-k_{D}(0,i(\alpha_k+1/2))\\
&  \quad \geq\sum_{j=1}^{k-1} \left[k_{S}(i(\alpha_j-1/2), i(\alpha_j+1/2))+k_{S}(i(\alpha_j-1/2), i(\alpha_{j+1}+1/2))\right.\\
& \quad\qquad \left.-k_{D}(i(\alpha_j-1/2), i(\alpha_j+1/2))-k_{D}(i(\alpha_j-1/2), i(\alpha_{j+1}+1/2))\right]\\
& \quad\geq  \sum_{j=1}^{k-1} \left[k_{S}(i(\alpha_j-1/2), i(\alpha_j+1/2))-k_{D}(i(\alpha_j-1/2), i(\alpha_j+1/2))\right]\\
&\quad \geq  \sum_{j=1}^{k-1} \left[k_{S}(i(\alpha_j-1/2), i(\alpha_j+1/2))-\left(1+\frac{1}{2j}\right)k_{D_j}(i(\alpha_j-1/2), i(\alpha_j+1/2))\right.]\\
&\quad \geq  \sum_{j=1}^{k-1} \left[k_{S}(i(\alpha_j-1/2), i(\alpha_j+1/2))-\left(1+\frac{1}{2j}\right)k_{S_j}(i(\alpha_j-1/2), i(\alpha_j+1/2))\right.]\\
&\quad =  \sum_{j=1}^{k-1}\left[\frac{\pi}{4}-\left(1+\frac{1}{2j}\right)\frac{\pi}{4\left(1+\frac{1}{j}\right)}\right]=\frac{\pi}{8}\sum_{j=1}^{k-1}\frac{1}{1+j},
\end{split}
\end{equation*}
and \eqref{no-reg3} holds. 

Therefore, $g$ is not regular at $\sigma$. Note that, by Remark \ref{unico-modello}, for every other pre-model $(\D, \tilde g, \eta_t)$ for $(\phi_t)$ at $\sigma$ the map $\tilde g$ is not regular at $\sigma$.

\end{document}